\pdfoutput=1
\documentclass[reqno]{amsart}

\usepackage{graphicx}
\usepackage[utf8]{inputenc}
\usepackage[T1]{fontenc}
\usepackage{siunitx}
\usepackage{amsmath}
\usepackage{amsfonts}
\usepackage{bm} % For proper bold symbols. \boldsymbol does not work with the mathptmx times font
\usepackage[colorlinks=true,linkcolor=blue,citecolor=blue]{hyperref}
\usepackage{placeins}

\newcommand{\diff}[1]{\, \mathrm{d}#1}

\DeclareMathOperator{\Dt}{D_t}
\DeclareMathOperator{\Dtt}{D_t^2}
\DeclareMathOperator{\trace}{tr}

\newcommand{\bA}{\mathbf{A}}
\newcommand{\bB}{\mathbf{B}}
\newcommand{\bH}{\mathbf{H}}
\newcommand{\bI}{\mathbf{I}}
\newcommand{\bK}{\mathbf{K}}
\newcommand{\bL}{\mathbf{L}}
\newcommand{\bM}{\mathbf{M}}
\newcommand{\bR}{\mathbf{R}}

\newcommand{\R}{\mathbb{R}}
\newcommand{\Vd}{\bm{V}}
\newcommand{\Sh}{S_h}
\newcommand{\Sdh}{\bm{S}_h}

\newcommand{\Chi}{\bm{\chi}}

\newcommand{\iprod}[2]{\left( #1, #2 \right)}
\newcommand{\norm}[1]{\lVert #1 \rVert}

\newcommand{\Bignorm}[1]{\Big\lVert #1 \Big\rVert}

\newcommand{\strain}[1]{\epsilon(#1)}

\newcommand{\ordo}{{\rm O}}

\newcommand{\tend}{T}
\newcommand{\eth}{e_{\theta}}
\newcommand{\ephi}{e_{\phi}}
\newcommand{\eu}{e_{u}}
\newcommand{\ethh}{e_{\theta,h}}
\newcommand{\ephih}{e_{\phi,h}}
\newcommand{\euh}{e_{u,h}}
\newcommand{\eh}{e_{h}}

\theoremstyle{theorem}
\newtheorem{assumption}{Assumption}[section]
\newtheorem{theorem}{Theorem}[section]
\theoremstyle{lemma}
\newtheorem{lemma}{Lemma}[section]
\theoremstyle{remark}
\newtheorem{remark}{Remark}[section]

\begin{document}

\date{Februart 20, 2017}

\title[FEM analysis for a thermoviscoelastic system]{Finite element convergence analysis for the thermoviscoelastic Joule heating problem }

\author{Axel M\r{a}lqvist}
\address{Mathematical Sciences\\Chalmers University of Technology and the University of Gothenburg\\,
     SE-412 96 Göteborg\\ Sweden.}
\email[A.~Målqvist]{axel@chalmers.se}

\author{Tony Stillfjord}
\email[T.~Stillfjord]{tony.stillfjord@gu.se}

\thanks{Both authors were supported by the Swedish Research Council under grant 2015-04964}

\begin{abstract}
  We consider a system of equations that model the temperature, electric potential and deformation of a thermoviscoelastic body. A typical application is a thermistor; an electrical component that can be used e.g.\ as a surge protector, temperature sensor or for very precise positioning. We introduce a full discretization based on standard finite elements in space and a semi-implicit Euler-type method in time. For this method we prove optimal convergence orders, i.e.\ second-order in space and first-order in time. The theoretical results are verified by several numerical experiments in two and three dimensions.
\end{abstract}



% 65-XX Numerical analysis
% -65Mxx Partial differential equations, initial value and time-dependent initial-boundary value problems
% --65M12 Stability and convergence of numerical methods
% --65M60 Finite elements, Rayleigh-Ritz and Galerkin methods, finite methods

% 74-XX Mechanics of deformable solids
% -74Dxx Materials of strain-rate type and history type, other materials with memory (including elastic materials with viscous damping, various viscoelastic materials)
% --74D05 Linear constitutive equations
% -74Hxx Dynamical problems
% --74H15 Numerical approximation of solutions

\maketitle

\section{Introduction}
Consider the following system of coupled equations:
\begin{align}
  \label{eq:main_T}
    \dot{\theta} &= \Delta \theta + \sigma(\theta) | \nabla \phi|^2 -  \bM : \strain{\dot{u}}, \\
  \label{eq:main_phi}
   0 &= \nabla \cdot \big( \sigma(\theta) \nabla \phi \big), \\
  \label{eq:main_stress}
    \ddot{u} &= \nabla \cdot \big( \bA\strain{\dot{u}} + \bB\strain{u} - \bM \theta \big) + f, 
\end{align}
with initial conditions
\begin{equation*}
  \theta(0,x) = \theta_0(x), \quad u(0,x) = u_0(x) \quad \text{and} \quad \dot{u}(0,x) = v_0(x),
\end{equation*}
over the convex polygonal or polyhedral domain $\Omega \subset \R^d$ with $d \le 3$.
Together with appropriate boundary conditions, to be specified later, these equations describe the evolution of the temperature $\theta$, electric potential $\phi$ and deformation $u$ of a conducting body. Here $\bA$, $\bB$ are constant tensors of order 4, describing the viscosity and elasticity of the body, and $\bM$ is a constant matrix describing the thermal expansion of the body. The vector $f$ consists of external forces and $\sigma(\theta)$ denotes the electrical conductivity, which here depends on the temperature. In addition, we have used the notation 
\begin{equation*}
  \strain{u} = \frac{1}{2} \big( \nabla u + (\nabla u)^T \big)
\end{equation*}
for the linearized strain tensor and $:$ for the Frobenius inner product.

The coupling of electricity and temperature through~\eqref{eq:main_T}--\eqref{eq:main_phi} is commonly known as \emph{Joule heating} and is typically used to model thermistors, see e.g.~\cite{AntontsevChipot1994,Cimatti1992}. These are electrical components used for example as surge protectors or temperature sensors. The inclusion of thermoviscoelastic effects through~\eqref{eq:main_stress} allows us to also model their use as actuators on the micro-scale, cf.~\cite{Henneken_etal2006}. 

We note that the Joule heating problem, both stationary and time-dependent, has been considered extensively in different contexts. For discussions on existence and uniqueness, see e.g.~\cite{AllegrettoXie1992,AntontsevChipot1994,Chen1994_2,Cimatti1989,Cimatti1992,Holst_etal2010,Howison_etal1993,JensenMalqvist2013,LiYang2015,YuanLiu1994} and the references therein. For the fully coupled, deformable problem the literature is less extensive. We refer mainly to~\cite{Kuttler_etal2008} for the non-degenerate case that we consider here, with $\sigma \ge \sigma_{\text{min}} > 0$. See also~\cite{WuXu2006} for the degenerate case where $\sigma = 0$ is allowed; this requires a more generalized solution concept.

However, to our knowledge there exists no numerical analysis for methods applied to the fully coupled case. Many authors have analyzed methods for similar problems. For example, \cite{Fernandez2006} considers the quasi-static version where the $\ddot{u}$-term is ignored, \cite{AkrivisLarsson2005}, \cite{ElliottLarsson1995} and \cite{LiGaoSun2014} considers the non-deformable case, \cite{FernandezKuttler2010_2,FernandezKuttler2010} treat the purely thermoviscoelastic case (no $\phi$) with nonlinear constituent law, etc. Additionally, in the deformable case a common theme seems to be suboptimal convergence orders, i.e.\ errors of the form $\ordo(h+k)$ instead of $\ordo(h^2+k)$. 

The main contribution of this article is therefore an error analysis for a fully discrete discretization applied to the problem~\eqref{eq:main_T}--\eqref{eq:main_stress}, which shows optimal convergence orders in both time and space. For the spatial discretization we consider standard finite elements, and for the temporal discretization a semi-implicit Euler-type method. Our approach also allows us to analyze e.g.\ the implicit Euler method, but the semi-implicit method benefits from a greatly decreased computational cost while the errors are comparable.

The central idea of our proof is to bound the errors in $\phi$ and $\dot{u}$ in terms of the error in $\theta$, in the spirit of~\cite{ElliottLarsson1995} and~\cite{LiSun2013}. The latter error then fulfills an equation similar to~\eqref{eq:main_T}, to which we may apply a Grönwall inequality after properly handling the quadratic potential term. We note that we avoid any time step restrictions of the form $k \le h^{d/r}$ by performing the analysis in two steps, where the first considers only the discretization in time, cf.~\cite{LiSun2013}.
Finally, in order to produce the $\dot{u}$ error bound, we extend the concept of Ritz-Volterra projections for damped wave equations (see~\cite{LinThomeeWahlbin1991}) to the discrete and vector-valued viscoelasticity case.

For simplicity, we consider Dirichlet boundary conditions,
\begin{equation*}
  \theta(t,x) = 0, \quad  \phi(t,x) = \phi_b(t,x) \quad \text{and} \quad u(t,x) = 0 
\end{equation*}
for $t \in [0, \tend]$ and $x \in \partial \Omega$. This is a simplified case of the ideal situation with an arbitrary polygon and mixed boundary conditions, corresponding to where the body is clamped and insulated. As is well known (see e.g.~\cite{Grisvard1985}) the solutions to such a problem would typically suffer from a lack of regularity in the vicinity of re-entrant corners and boundary condition transitions, which leads to suboptimal convergence orders for finite-element based numerical methods. We therefore restrict ourselves to the simplified model, and will indicate possible generalizations by our numerical experiments. 

A brief outline of the article is as follows. In Section~\ref{sec:problem} we write the problem on weak form and discretize it in both time and space. The assumptions on the data and solutions to the continuous problem are given in Section~\ref{sec:error_analysis}, where we also perform the error analysis. In Subsection~\ref{subsec:timediscrete}, the time-discrete system is shown to be first-order convergent, and then the full discretization is shown to be second-order convergent to the time-discrete system in Subsection~\ref{subsec:fully_discrete}. These results are confirmed by the numerical experiments presented in Section~\ref{sec:experiments}, and conclusions and future work is summarized in Section~\ref{sec:conclusions}.

\section{Weak formulation and discretization} \label{sec:problem}
In order to present a weak formulation of the problem, we introduce the spaces
\begin{equation*}
  V := H^1_0(\Omega) \subset L^2(\Omega),  \quad  \text{and } \quad \Vd := H^1_0(\Omega)^d \subset L^2(\Omega)^d =: \bL^2(\Omega),
\end{equation*}
as well as the space of symmetric matrices,
\begin{equation*}
  Q = \{\xi = (\xi_{ij})_{i,j=1}^d \subset L^2(\Omega)^{d \times d} \; ; \; \xi_{ji} = \xi_{ij}, 1 \le i, j \le d \}.
\end{equation*}
 The idea here is that $\theta$ and $\phi-\phi_b$ belong to $V$, $u \in \Vd$ and $\strain{u} \in Q$. On $Q$, we have the inner product 
 \begin{equation*}
   \iprod{\xi}{\zeta}_Q := \int_{\Omega}{\xi(x) : \zeta(x) \diff{x}} = \sum_{i,j = 1}^d {\iprod{\xi_{ij}}{\zeta_{ij}}_{L^2(\Omega)}}.
 \end{equation*}
 which gives rise to the norm $\norm{\cdot}_Q$. To simplify some notation, we use the inner product
 \begin{equation*}
   \iprod{u}{v}_{\Vd} = \iprod{\strain{u}}{\strain{v}}_Q 
 \end{equation*}
 on $\Vd$ instead of the usual one. 
The norm $\norm{\cdot}_{\Vd}$ induced by this inner product is equivalent to $\norm{\cdot}_{H^1(\Omega)^d}$ by Korn's inequality, see e.g.~\cite[Chapter III, Theorems 3.1, 3.3]{DuvautLions1972} and~\cite{Nitsche1981}. 
We will on several occasions make use also of the norm $\norm{\cdot}_{\bB}$, which arises from the elasticity operator through
\begin{equation*}
  \norm{u}_{\bB}^2 = \iprod{\bB\strain{u}}{\strain{u}}_Q,
\end{equation*}
as well as the norm $\norm{\cdot}_{\bA + k\bB}$ defined analogously for a small positive constant $k$. Under Assumption~\ref{ass:viscoelasticity} in the next section, both of these norms are equivalent to the $\Vd$-norm.
In the following, we will omit the specification of $\Omega$ and simply write $L^2$ or $\bL^2$. Additionally, the $L^2$- and $\bL^2$-norms will both simply be denoted by $\norm{\cdot}$ and the corresponding inner products by $\iprod{\cdot}{\cdot}$, where no confusion can arise. 

By multiplying the equations~\eqref{eq:main_T},~\eqref{eq:main_phi} with the test function $\chi \in V$, Equation~\eqref{eq:main_stress} with $\Chi \in \Vd$ and then using Green's formula we get
\begin{align}
  \label{eq:weak_T}
    \iprod{\dot{\theta}}{\chi} + \iprod{\nabla \theta}{\nabla \chi} &= \iprod{\sigma(\theta) | \nabla \phi|^2}{\chi} - \iprod{ \bM : \strain{\dot{u}}}{\chi}, \\
  \label{eq:weak_phi}
    \iprod{\sigma(\theta) \nabla \phi}{\nabla \chi} &= 0,\\
  \label{eq:weak_stress}
    \iprod{\ddot{u}}{\Chi} + \iprod{\bA\strain{\dot{u}} + \bB\strain{u}}{\strain{\Chi}}_Q &= \iprod{\bM \theta }{\strain{\chi}}_Q + \iprod{f}{\Chi},
\end{align}
for all $\chi \in V$ and $\Chi \in \Vd$, respectively. In~\eqref{eq:weak_stress}, we have made use of the identity ${\iprod{\strain{u}}{\nabla v} = \iprod{\strain{u}}{\strain{v}}}$ as well as the similar identities ${\iprod{\bA\strain{u}}{\nabla v} = \iprod{\bA\strain{u}}{\strain{v}}}$ and ${\iprod{\bB\strain{u}}{\nabla v} = \iprod{\bB\strain{u}}{\strain{v}}}$. The latter two hold because we assume $\bA$ and $\bB$ to be symmetric; see Assumption~\ref{ass:viscoelasticity} in the next section. Note also that we have omitted the time parameter here and in the original equation; both are supposed to hold for all times $t \in (0,\tend]$ for a given $\tend$.

We now discretize the time interval $[0,\tend]$ using a constant temporal step size $k$, which results in the grid $t_n = nk$ with $n = 1, 2, \ldots, N$ and $Nk = \tend$. We will abbreviate function evaluations at these times by sub-scripts, so that 
\begin{equation*}
  \theta_n = \theta(t_n), \quad  \phi_n = \phi(t_n), \quad u_n = u(t_n)  \quad \text{and} \quad f_n = f(t_n).
\end{equation*}
 The approximations of these solution values should belong to the same spaces as in the continuous case, and we will denote them by capital letters and superscripts: 
\begin{equation*}
  \Theta^n \approx \theta_n, \quad \Phi^n \approx \phi_n \quad \text{and} \quad U^n \approx u_n .
\end{equation*}
Additionally, we denote by $\Dt$ the first-order backward difference quotient, i.e.\
\begin{equation*}
  \Dt \Theta^n = \frac{\Theta^n - \Theta^{n-1}}{k}.
\end{equation*}

With this notation given, we now consider the following semi-implicit temporal discretization of Equations~\eqref{eq:main_T}--\eqref{eq:main_stress}, 
\begin{align}
  \label{eq:semi_T}
    \Dt{\Theta^n} &= \Delta \Theta^n + \sigma(\Theta^{n-1}) | \nabla \Phi^{n-1}|^2 -  \bM : \strain{\Dt U^{n-1}}, \\
  \label{eq:semi_phi}
   0 &= \nabla \cdot \big( \sigma(\Theta^n) \nabla \Phi^n \big), \\
  \label{eq:semi_stress}
    \Dtt U^n &= \nabla \cdot \big( \bA\strain{\Dt U^n} + \bB\strain{U^n} - \bM \Theta^n \big) + f_n, 
\end{align}
where $\Dtt = \Dt\Dt$, and its corresponding weak form, 
\begin{align}
  \label{eq:semi_weak_T}
    &\iprod{\Dt \Theta^n}{\chi} + \iprod{\nabla \Theta^n}{\nabla \chi} = \iprod{\sigma(\Theta^{n-1}) | \nabla \Phi^{n-1}|^2}{\chi} - \iprod{ \bM : \strain{\Dt U^{n-1}}}{\chi}, \\
  \label{eq:semi_weak_phi}
    & \hspace{2.2cm} \iprod{\sigma(\Theta^n) \nabla \Phi^n}{\nabla \chi} = 0,\\
  \label{eq:semi_weak_stress}
    &\iprod{\Dtt U^n}{\Chi} + \iprod{\bA\strain{\Dt U^n} + \bB\strain{U^n}}{\strain{\Chi}}_Q = \iprod{\bM \Theta^n}{\strain{\chi}}_Q + \iprod{f_n}{\Chi},
\end{align}
for $n = 1, \ldots, N$ and for all $\chi \in \Sh$ and $\Chi \in \Sdh$, respectively. The initial conditions are the same as in the continuous case: $\Theta^0 = \theta_0$, $U^0 = u_0$ and $\Dt U^0 = v_0$. (We use a fictitious point $U^{-1}$ to define $\Dt U^0$.)
Note that this discretization results in a decoupling of the equations; we solve first for $\Theta^n$ using~\eqref{eq:semi_T} then use this to find $\Phi^n$ from~\eqref{eq:semi_phi} and $U^n$ from~\eqref{eq:semi_stress}. This implies a significant decrease in computational effort compared to the fully coupled case arising from e.g.\ the implicit Euler discretization.

For the spatial discretization, we introduce the finite element spaces $\Sh \subset V$ and $\Sdh \subset \Vd$. These consist of continuous, piecewise linear functions with zero trace on $\partial \Omega$, defined on a quasi-uniform mesh with mesh-width $h$.
Then the fully discrete problem we are interested in is given by
\begin{align}
  \label{eq:full_T}
    &\iprod{\Dt \Theta_h^n}{\chi} + \iprod{\nabla \Theta_h^n}{\nabla \chi} = \iprod{\sigma(\Theta_h^{n-1}) | \nabla \Phi_h^{n-1}|^2}{\chi} - \iprod{ \bM : \strain{\Dt U_h^{n-1}}}{\chi}, \\
  \label{eq:full_phi}
    & \hspace{2.2cm} \iprod{\sigma(\Theta_h^n) \nabla \Phi_h^n}{\nabla \chi} = 0,\\
  \label{eq:full_stress}
    &\iprod{\Dtt U_h^n}{\Chi} + \iprod{\bA\strain{\Dt U_h^n} + \bB\strain{U_h^n}}{\strain{\Chi}}_Q = \iprod{\bM \Theta_h^n}{\strain{\chi}}_Q + \iprod{f_n}{\Chi},
\end{align}
for $n = 1, \ldots, N$ and for all $\chi \in \Sh$ and $\Chi \in \Sdh$, respectively. Here, the approximations satisfy $\Theta_h^n \in \Sh$, $\Phi_h^n - \phi_b(t_n) \in \Sh$ and $U_h^n \in \Sdh$. (We assume that $\phi_b(t_n)$ is defined on all of $\Omega$.)
As initial conditions, we take $U_h^0 = 0$, $\Dt U_h^0 = 0$ and $\Theta_h^0 = I_h\theta_0$, the Lagrangian interpolant of the exact initial condition.

\begin{remark}
  We assume the domain to be a convex polygon or polyhedron in order that the standard interpolation and regularity estimates for linear elliptic problems are satisfied, see~\cite[Section 3.2]{Ciarlet2002}. Similarly, the quasi-uniformity of the mesh guarantees that the standard inverse inequalities are satisfied. These are needed to handle the nonlinear potential term in~\eqref{eq:main_T}, see~\cite{ElliottLarsson1995,LiSun2013}.
\end{remark}

\section{Error analysis} \label{sec:error_analysis}
Our main goal is to estimate the errors $\norm{\Theta_h^n - \theta_n}$, $\norm{\Phi_h^n - \phi_n}$ and $\norm{U_h^n - u_n}$.
In order to do this, we will generalize the analysis of \cite{LiSun2013} (cf.\ also~\cite{ElliottLarsson1995}) for the case with no deformation. This consists of first showing that the time-discrete approximations are $\ordo(k)$-close to the solutions of the continuous system, and also proving that these approximations exhibit a certain regularity. The key part here is to express the error in the potential in terms of the error in the temperature, and then only working with the temperature equation. With the given regularity, the time-discrete and fully discrete approximations can then be compared and shown to be $\ordo(h^2)$-close. The main problem here is the nonlinear term $\sigma(\theta) |\nabla \phi|^2$, which is handled in a two-step fashion: first using that $\norm{\nabla (\Phi_h^n - \Phi^n)} \le C(h + \norm{\Theta_h^n - \Theta^n})$ to show that in fact $\norm{\nabla (\Phi_h^n - \Phi^n)} \le Ch$ and then using this to estimate ${\nabla (\Phi_h^n - \Phi^n)}$ in a stronger norm.

In our case, the temperature equation~\eqref{eq:main_T} contains the extra term $\bM : \strain{\dot{u}}$, so our idea is to also bound the error in $\dot{u}$ by the error in the temperature. Then we show that the approximations $U^n$ possess certain regularity, which may be used to also express the fully discrete deformation errors in terms of the fully discrete temperature errors. The key part in the latter step is to utilize the concept of Ritz-Volterra projections~\cite{LinThomeeWahlbin1991}, which we here generalize to the vector-valued viscoelasticity case, as well as to discrete time.

Before we perform this extended analysis, we state the general assumptions on the given data. In these, as well as throughout the rest of the paper, $C$ denotes a generic constant independent of $k$, $h$ and $n$ but possibly depending on $\tend$, that may differ from line to line.

\begin{assumption}\label{ass:viscoelasticity}
  The viscosity and elasticity tensors $\bA = (a_{ijkl})$ and $\bB = (b_{ijkl})$ are symmetric, and both yield Lipschitz continuous and strongly coercive bilinear forms.  That is, 
  \begin{equation*}
    a_{ijkl} = a_{jikl} = a_{klij}, \qquad b_{ijkl} = b_{jikl} = b_{klij},
  \end{equation*}
and there are positive constants $C_1, C_2$ such that for all $u, v \in \Vd$ we have
  \begin{align*}
    \max \Big( \iprod{\bA \strain{u}}{\strain{v}}_Q , \iprod{\bB \strain{u}}{\strain{v}}_Q \Big) &\le C_1 \norm{u}_{\Vd}\norm{v}_{\Vd} \quad \text{and}\\
   \min \Big( \iprod{\bA \strain{u}}{\strain{u}}_Q, \iprod{\bB \strain{u}}{\strain{u}}_Q \Big) &\ge C_2 \norm{u}_{\Vd}^2. 
  \end{align*}
\end{assumption}

\begin{assumption}\label{ass:conductivity}
  The electrical conductivity $\sigma$ belongs to $C^1(\R)$ and there are positive constants $\sigma_{\text{min}}$, $\sigma_{\text{max}}$ and $\sigma_{\text{max}}'$ such that for all $\theta \ge 0$ we have
  \begin{equation*}
    0 < \sigma_{\text{min}} \le \sigma(\theta)\le \sigma_{\text{max}} \quad \text{and} \quad |\sigma'(\theta)| \le \sigma_{\text{max}}'.
  \end{equation*}
\end{assumption}

\begin{assumption}\label{ass:data}
The function $f \in C(0, \, \tend; \, \bL^2)$, $\theta_0 \in H^2 \cap H^1_0$ and $\phi_b \in L^{\infty}(0, T; L^2)$ is regular enough that
\begin{equation*}
  \norm{\phi_b}_{L^{\infty}(0, \, \tend; \, W^{2, 12/5})} + \norm{\dot{\phi}_b}_{L^{2}(0, \, \tend; \, H^1)} + \norm{\nabla \phi_b}_{L^{\infty}(0, \, \tend; \, L^\infty)} \le C.
\end{equation*}
\end{assumption}

By~\cite{Kuttler_etal2008}, these assumptions guarantee the existence of a weak solution to the problem, i.e functions $(\theta, \phi, u)$ satisfying~\eqref{eq:weak_T}--\eqref{eq:weak_stress} with the time derivatives interpreted in a weak sense. Thus for example $\theta \in L^2(0, \tend; V)$ and $\dot{\theta} \in L^2(0, \tend; V)'$. For optimal convergence orders more regularity is required, and explicit conditions on the data that guarantees such regularity is currently unknown. We therefore also make the following regularity assumption, where $\bH^2 = H^2(\Omega)^d$:

\begin{assumption} \label{ass:regularity}
  There exist solutions $(\theta, \phi, u)$ to~\eqref{eq:weak_T}--\eqref{eq:weak_stress} over the time interval $[0, \tend]$ which are regular enough that
  \begin{align*}
    \norm{\theta}_{L^{\infty}(0, \, \tend; \, H^2)} + \norm{\dot{\theta}}_{L^{\infty}(0, \, \tend; \, L^2)} + \norm{\dot{\theta}}_{L^{2}(0, \,  \tend; \, H^2)} + \norm{\ddot{\theta}}_{L^{1}(0, \, \tend; \, L^2)} &\le C, \\
    \norm{\phi}_{L^{\infty}(0, \, \tend; \, W^{2,12/5})} + \norm{\dot{\phi}}_{L^{2}(0, \, \tend; \, H^1)} + \norm{\phi}_{L^{\infty}(0, \, \tend; \, W^{1,\infty})} &\le C, \\
    \norm{\dot{u}}_{L^{\infty}(0, \, \tend; \, \bH^2)} + \norm{\ddot{u}}_{L^{\infty}(0, \, \tend; \, \bH^2 )} + \norm{u^{(3)}}_{L^{1}(0, \, \tend; \, \bL^2)}&\le C
  \end{align*}
\end{assumption}
The assumptions on $\theta$ and $\phi$ are essentially the same as in the non-deformable situation given in~\cite{LiSun2013}, while the assumptions on $u$ and $f$ are new. 
We note that for the non-deformable case, the existence of solutions with similar regularity properties was shown in \cite{ElliottLarsson1995} when $d \le 2$, with weak requirements on the initial values. In the general elliptic/parabolic case, the absence of reentrant corners in the convex domain makes such regularity plausible, see e.g.~\cite[Chapters 3,4]{Grisvard1985} and~\cite[Chapter 19]{Thomee2006}. In the displacement equation the viscosity term acts as damping, and we expect regular solutions to be present also there, see e.g.~\cite{LarssonThomeeWahlbin1991}. We are not aware of any regularity results for the fully coupled system, but 
we note that our numerical experiments with smooth data suggest that Assumption~\ref{ass:regularity} is satisfied in practice.

The following main theorem will be proved in the next two subsections:
\begin{theorem} \label{theorem:main}
  Let Assumptions~\ref{ass:viscoelasticity}-\ref{ass:regularity} be satisfied and let $(\theta, \phi, u)$ and $(\Theta_h^n, \Phi_h^n, U_h^n)$ be solutions to the equations~\eqref{eq:weak_T}--\eqref{eq:weak_stress} and~\eqref{eq:full_T}--\eqref{eq:full_stress}, respectively. Then there are positive constants $k_0$ and $h_0$ such that if $k < k_0$ and $h < h_0$ we have for $n = 1, \ldots, N$ that
  \begin{equation*}
    \norm{\Theta_h^n - \theta_n} +  \norm{\Phi_h^n - \phi_n} + \norm{\Dt U_h^n - \dot{u}_n} \le C(h^2 + k),
  \end{equation*}
and
\begin{equation*}
  \norm{\Theta_h^n - \theta_n}_{H^1} + \norm{\Phi_h^n - \phi_n}_{H^1} + \norm{\Dt U_h^n - \dot{u}_n}_{\Vd} \le C(h + k).
\end{equation*}
The constant $C$ is independent of $k$, $h$ and $n$, but may depend on the final time $\tend = Nk$ and the problem data.
\end{theorem}
To abbreviate expressions like the above in the following, we introduce 
\begin{equation*}
  \eth^n = \Theta^n - \theta_n, \quad \ephi^n = \Phi^n - \phi_n \quad \text{and} \quad \eu^n = U^n - u_n
\end{equation*}
as well as
\begin{equation*}
  \ethh^n = \Theta_h^n - \Theta^n, \quad \ephih^n = \Phi_h^n - \Phi^n \quad \text{and} \quad \euh^n = U_h^n - U^n.
\end{equation*}

\subsection{The time-discrete case} \label{subsec:timediscrete}
We start by considering the semi-discrete case, and first provide a bound for $\Dt \eu^n$ in terms of $\eth^n$. 

\begin{lemma} \label{lemma:DtU_bound}
Let Assumptions~\ref{ass:viscoelasticity}-\ref{ass:regularity} be satisfied and let $(\theta, \phi, u)$ and $(\Theta^n, \Phi^n, U^n)$ be solutions to the equations~\eqref{eq:weak_T}--\eqref{eq:weak_stress} and~\eqref{eq:semi_weak_T}--\eqref{eq:semi_weak_stress}, respectively. Then we have
  \begin{equation*} 
    \norm{\Dt \eu^n}^2 + \norm{\eu^n}_{\Vd}^2 + k \sum_{j=1}^{n}{\norm{\Dt \eu^j}_{\Vd}^2 } \le Ck^2 + Ck \sum_{j=1}^{n}{\norm{\eth^j}^2 },
  \end{equation*}
for $n = 1, \ldots, N$, with the constant $C$ independent of $k$ and $n$.
\end{lemma}

\begin{proof}
By equations~\eqref{eq:weak_stress} and~\eqref{eq:semi_weak_stress}, we see that the error $\eu^n$ satisfies
\begin{align*}
  \iprod{\Dtt \eu^n}{\Chi} + \iprod{\bA \strain{\Dt\eu^n} + \bB \strain{\eu^n}}{\strain{\Chi}} &=  \iprod{\bM \eth^n}{\strain{\Chi}} + \iprod{\ddot{u}(t_n) - \Dtt u(t_n)}{\Chi} \\
 &+  \iprod{\bA \strain{\dot{u}(t_n) - \Dt u(t_n)}}{\strain{\Chi}}\\
  &\le C\norm{\eth^n} \norm{\Chi}_{\Vd} + Ck\norm{\Chi} + Ck\norm{\Chi}_{\Vd} 
\end{align*}
due to the regularity assumptions on $u$.
We note that for any sequence $\{g^n\}$ we have 
\begin{equation*}
2\iprod{\Dt^2 g^n}{\Dt g^n} \ge \Dt \norm{\Dt g^n}^2  \quad \text{and} \quad 2\iprod{\bB \strain{g^n}}{\strain{\Dt g^n}} \ge \Dt \norm{g^n}_{\bB}^2,
\end{equation*}
where $\norm{\cdot}_{\bB}$ is the norm induced by the inner product $\iprod{\bB \strain{\cdot}}{\strain{\cdot}}$. Thus by choosing $\Chi = \Dt \eu^n$ and using the Cauchy--Schwarz inequality as well as Young's inequality, $ab \le \frac{1}{2c}a^2 + \frac{c}{2}b^2$, we get
\begin{equation*}
  \Dt \norm{\Dt \eu^n}^2 +  2C_2 \norm{\Dt \eu^n}_{\Vd} + \Dt \norm{\eu^n}_{\bB}^2\le Ck^2 + C\norm{\eth^n}^2 + C_2\norm{\Dt \eu^n}_{\Vd}^2 .
\end{equation*}
Canceling the final term, summing over $n$ and modifying the constants then yields
\begin{equation*}
  \norm{\Dt \eu^n}^2 +  k\sum_{j=1}^{n}{ \norm{\Dt \eu^j}_{\Vd}} + \norm{\eu^n}_{\bB}^2\le Ck^2 + Ck\sum_{j=1}^{n}{\norm{\eth^j}^2},
\end{equation*}
and the Lemma follows from the equivalence between the $\bB$- and $\Vd$-norms.

\end{proof}

\begin{theorem} \label{theorem:time_error}
  Let Assumptions~\ref{ass:viscoelasticity}-\ref{ass:regularity} be satisfied and let $(\theta, \phi, u)$ and $(\Theta^n, \Phi^n, U^n)$ be solutions to the equations~\eqref{eq:main_T}--\eqref{eq:main_stress} and~\eqref{eq:semi_T}--\eqref{eq:semi_stress}, respectively. Then there is a positive constant $k_0$ such that if $k < k_0$ then
  \begin{equation*}
    \norm{\eth^n}_{H^1}^2 + \norm{\ephi^n}_{H^1}^2 + \norm{\Dt \eu^n}_{\Vd}^2 \le Ck^2,
  \end{equation*}
for $n = 1, \ldots, N$, with the constant $C$ independent of $k$ and $n$. In addition, the approximations have the following regularity:
\begin{align*}
  \norm{\Theta^n}_{H^2}^2 + \norm{\Dt \Theta^n}^2 + k\sum_{j=1}^{n}{\norm{\Dt \Theta^j}_{H^2}^2} &\le C,\\
\norm{\Phi^n}_{W^{2, 12/5}} + \norm{\Phi^n}_{W^{1,\infty}} &\le C, \\
\norm{\Dt U^n}_{\bH^2}^2 + \norm{\Dtt U^n}_{\Vd}^2 + k\sum_{j=1}^{n}{\norm{\Dtt U^j}_{\bH^2}^2} &\le C.
\end{align*}
\end{theorem}

\begin{proof}
To begin with, we see that the error $\ephi^n$ satisfies
\begin{equation*}
  - \nabla \cdot \big(\sigma(\Theta^n) \nabla \ephi^n)\big) = \nabla \cdot \big( (\sigma(\Theta^n) - \sigma(\theta_n)) \nabla \phi_n \big).
\end{equation*}
Multiplying this equation by $\ephi^n$ and integrating directly yields
\begin{equation*}
  \norm{\nabla \ephi^n}^2 \le C \norm{\nabla \phi_n}_{L^{\infty}} \norm{\eth^n} \norm{\nabla \ephi^n} ,
\end{equation*}
so that
\begin{equation} \label{eq:ephi_bound}
  \norm{\nabla \ephi^n} \le C \norm{\eth^n}
\end{equation}
by the regularity assumptions. This inequality for $\ephi^n$ corresponds to Lemma~\ref{lemma:DtU_bound} for $\eu^n$. Further, we see that the error $\eth^n$ satisfies 
\begin{equation}\label{eq:eth_equation}
\begin{aligned}
  \Dt \eth^n - \Delta \eth^n &= \Big( \sigma(\Theta^{n-1}) - \sigma(\theta_{n-1}) \Big) |\nabla \phi_{n-1}|^2 + \sigma(\Theta^{n-1}) \Big( \nabla \Phi^{n-1} + \nabla \phi_{n-1}\Big) \cdot  \nabla \ephi^{n-1} \\
                             &\quad - M : \strain{\Dt \eu^{n-1}} + R_{\theta}^n, 
\end{aligned}
\end{equation}
where
\begin{align*}
   R_{\theta}^n &= \big( \sigma(\theta_{n-1}) - \sigma(\theta_{n}) \big) |\nabla \phi_{n-1}|^2 + \sigma(\theta_{n}) \big( \nabla \phi_{n-1} + \nabla \phi_{n}\big) \cdot \big( \nabla \phi_{n-1} - \nabla \phi_n\big) \\
              &\quad + M : \strain{\dot{u}_{n} - \dot{u}_{n-1}} + M : \strain{\dot{u}_{n-1} - \Dt u_{n-1}} .
\end{align*}
is bounded by $\norm{R_{\theta}^n} \le Ck$, again by the regularity assumptions. After multiplying by $\eth^n$ and integrating, we therefore get
\begin{equation} \label{eq:eth1}
\begin{aligned}
  \Dt \norm{\eth^n}^2 + 2\norm{\nabla \eth^n}^2 &\le C\norm{\eth^{n-1}} \norm{\eth^n} \norm{\nabla \phi_{n-1}}_{L^{\infty}} + \iprod{M : \strain{\Dt \eu^{n-1}}}{\eth^n} + Ck\norm{\eth^n}\\
&\quad + \iprod{\sigma(\Theta^{n-1}) \big( \nabla \Phi^{n-1} + \nabla \phi_{n-1}\big) \eth^n  }{\nabla \ephi^{n-1}}.
\end{aligned}
\end{equation}
The last term of this expression can be shown to be bounded by ${C(\norm{\eth^n}^2 + \norm{\ephi}_{H^1}^2)}$, see~\cite[p.627]{LiSun2013}, and for the second term we observe that for a generic $u \in \Vd$,
\begin{equation*}
   \iprod{\bM:(\nabla u)}{\chi}_{L^2} = \iprod{\nabla u}{\bM\chi}_Q = - \iprod{u}{\nabla\cdot(\bM\chi)}_{\bL^2} = - \iprod{u}{ \bM \nabla \chi}_{\bL^2}. 
\end{equation*}
As a completely analogous calculation holds also for $(\nabla u)^T$ and $\bM$ is symmetric, we thus have
\begin{equation} \label{eq:strain_by_parts}
  \iprod{\bM : \strain{u}}{\chi} = -\iprod{u}{\bM \nabla \chi} \le C\norm{u} \norm{\nabla \chi}.
\end{equation}
This implies that~\eqref{eq:eth1} reduces to
\begin{equation*}
  \Dt \norm{\eth^n}^2 + 2\norm{\nabla \eth^n}^2 \le C\big(k^2 + \norm{\eth^{n-1}}^2 + \norm{\eth^n}^2 + \norm{\ephi^{n-1}}_{H^1}^2 + \norm{\Dt \eu^{n-1}}^2\big) + \norm{\nabla \eth^n}^2.
\end{equation*}
Canceling the last term, summing up and using Equation~\eqref{eq:ephi_bound} and Lemma~\ref{lemma:DtU_bound} thus yields
\begin{equation*}
  \norm{\eth^n}^2 + k\sum_{j=1}^{n}{\norm{\nabla \eth^j}^2} \le Ck^2 + Ck\sum_{j=1}^{n}{\norm{\eth^{j}}^2}.
\end{equation*}
Under the step size restriction $Ck < 1$, we can eliminate the last term of the sum. An application of Grönwall's lemma then shows that the left-hand side is bounded by $Ck^2$. Using Equation~\eqref{eq:ephi_bound} and Lemma~\ref{lemma:DtU_bound} again, we see that in fact
\begin{equation*}
  \norm{\eth^n}^2 + k\sum_{j=1}^{n}{\norm{\nabla \eth^j}^2} + \norm{\nabla \ephi^n}^2 + \norm{\Dt \eu^n}^2 + \norm{\eu^n}_{\Vd}^2 + k \sum_{j=1}^{n}{\norm{\Dt \eu^j}_{\Vd}^2 } \le Ck^2  
\end{equation*}
From these preliminary bounds, we may deduce the desired regularity of $\Theta^n$ and $\Phi^n$ and then test~\eqref{eq:eth_equation} with $-\Delta \eth^n$ to acquire
\begin{equation*}
  \norm{\eth^n}_{H^1}^2 + k\sum_{j=1}^{n}{\norm{\Delta \eth^j}^2} \le Ck^2.
\end{equation*}
For details, we refer to~\cite[Theorem 3.1]{LiSun2013}. Let us instead investigate the remaining questions of the regularity of $U^n$ and the pointwise bound for $\Dt \eu^n$ in the $\Vd$-norm. By the defining equation, we have that
\begin{equation}\label{eq:euH2}
\begin{aligned} 
 \nabla \cdot \big( \bA\strain{\Dt \eu^n} + \bB\strain{\eu^n} \big) &=  \Dtt \eu^n + \nabla \cdot \big( \bM \Theta^n \big) 
                                                                    + \Dt^2 u(t_n) - \ddot{u}(t_n) \\
&\quad + \nabla \cdot \big(\bA \strain{\Dt u(t_n) - \dot{u}(t_n)} \big), 
\end{aligned}
\end{equation}
where the right-hand side is in $\bL^2$ since $\norm{\Dtt \eu^n} \le k^{-1}(\norm{\Dt \eu^n} + \norm{\Dt \eu^{n-1}}) \le C$. Let us denote it by $g_n$. Then we can rewrite the previous equation as
\begin{equation*} 
 \nabla \cdot \big( \bA\strain{\Dt \eu^n} + k\bB\strain{\Dt\eu^n} \big) =  g_n + \nabla \cdot \big(\bB\strain{\eu^{n-1}} \big).
\end{equation*}
Now since both $\bB$ and $\bA + k\bB$ induce bounded and coercive inner products on $\Vd$, we see that 
\begin{align*}
  \norm{\Dt \eu^n}_{\bH^2}^2 &\le C \norm{\nabla \cdot \big( \bA\strain{\Dt \eu^n} + k\bB\strain{\Dt\eu^n} \big)}^2 \\
                          &\le C\norm{g_n}^2 + C \norm{\eu^{n-1}}_{\bH^2}^2
\end{align*}
But since $\eu^{n-1} = k\sum_{j=1}^{n-1}{\Dt \eu^j}$, we can estimate the second term by Cauchy--Schwarz as
\begin{equation*}
  \norm{\eu^{n-1}}_{\bH^2}^2 \le k \sum_{j=1}^{n-1}{\norm{\Dt \eu^j}_{\bH^2}^2}.
\end{equation*}
An application of Grönwall's lemma thus shows that 
\begin{equation*}
  \norm{\Dt \eu^n}_{\bH^2} \le C,
\end{equation*}
which also implies that $\eu^n$, $U^n$ and $\Dt U^n$ are all in $\bH^2$. We may now multiply~\eqref{eq:euH2} by $\nabla \cdot \big( (\bA + k\bB)\strain{\Dt\eu^n} \big)$ and integrate to get 
\begin{equation*}
  \iprod{\Dt \strain{\Dt \eu^n}}{(\bA + k\bB)\strain{\Dt \eu^n}} + \norm{\nabla \cdot \big( (\bA + k\bB)\strain{\Dt\eu^n}\big)}^2 \le C \norm{\eth^n}_{\bH^1}^2 + C\norm{\eth^{n-1}}_{\bH^2}^2 ,
\end{equation*}
where we have used the Cauchy-Schwarz and Young inequalities and canceled a term $\frac{1}{2} \norm{\nabla \cdot \big( (\bA + k\bB)\strain{\Dt\eu^n}\big)}^2$. The first term on the left-hand side can be estimated from below by $\Dt \norm{\Dt \eu^n}_{\bA + k\bB}$, so summing up and using the equivalence of the ${(\bA + k\bB)}$- and $\Vd$-norms, we get
\begin{equation*}
  \norm{\Dt \eu^n}_{\Vd}^2 + k \sum_{j=1}^{n}{\norm{\Dt \eu^j}_{\bH^2}^2} \le C k \sum_{j=1}^{n-1}{\norm{\eth^j}_{H^1}^2} + Ck \sum_{j=1}^{n-1}{\norm{\Dt \eu^j}_{\bH^2}^2}.
\end{equation*}
But the first term in the right-hand side is bounded by $Ck^2$ and in the second term we may again use that $\norm{\Dt \eu^j}_{\bH^2}^2 \le k \sum_{i=1}^{j}{ \norm{\Dt \eu^i}_{\bH^2}^2}$. Defining
\begin{equation*}
  w_n = \norm{\Dt \eu^n}_{\Vd}^2 + k \sum_{j=1}^{n}{ \norm{\Dt \eu^j}_{\bH^2}^2},
\end{equation*}
we thus have
\begin{equation*}
  w_n \le Ck^2 + Ck \sum_{j=1}^{n-1}{w_j},
\end{equation*}
and an application of Grönwall's lemma shows that $w_n \le Ck^2$. This yields the final desired error bound, and additionally shows that $\norm{\Dt^2 \eu^n}_{\Vd}^2 + k \sum_{j=1}^{n}{ \norm{\Dt^2 \eu^j}_{\bH^2}^2} \le C$, which implies the stated regularity for $U^n$.

\end{proof}

\subsection{The fully discrete case} \label{subsec:fully_discrete}
We now turn to the fully discretized case and first prove an analogue to Lemma~\ref{lemma:DtU_bound}.

\begin{lemma} \label{lemma:RitzVolterra}
Let Assumptions~\ref{ass:viscoelasticity}-\ref{ass:regularity} be satisfied and $(\Theta^n, \Phi^n, U^n)$ and $(\Theta_h^n, \Phi_h^n, U_h^n)$ be solutions to equations~\eqref{eq:semi_weak_T}--\eqref{eq:semi_weak_stress} and~\eqref{eq:full_T}--\eqref{eq:full_stress}, respectively. Then there is a positive constant $k_0$ such that if $k < k_0$ we have for $n = 1, \ldots, N$ that
\begin{align*} 
  \norm{\euh^n}^2 + \norm{\Dt \euh^n}^2 &\le Ch^4 + Ck \sum_{j=1}^{n}{\norm{\ethh^j}^2 } \quad \text{and} \\
  \norm{\euh^n}_{\Vd}^2 + k \sum_{j=1}^{n}{\norm{\Dt \euh^j}_{\Vd}^2} &\le Ch^2 + Ck \sum_{j=1}^{n}{\norm{\ethh^j}^2 },
\end{align*}
with the constant $C$ independent of $k$, $h$ and $n$.
\end{lemma}
\begin{remark}
  In the case of a first-order equation, one would typically first add and subtract the Ritz projection of $\eu^n$ in order to work only in the finite element space. This approach is viable also in the second-order case, if one defines the Ritz projection using the $\iprod{\bA\strain{\cdot}}{\strain{\cdot}}$ inner product. We refer to~\cite{ThomeeZhang1989} for the scalar-valued case. However, we choose to instead work with a Ritz-Volterra projection, see~\cite{LinThomeeWahlbin1991} for the scalar-valued case. Such a projection takes both the $\bA$- and $\bB$-terms into account simultaneously, i.e.\ it is a projection of $C^1(0, \tend; \Vd)$-functions rather than of elements in $\Vd$. In the present situation, we need of course to consider a discretized version, but it nevertheless simplifies matters.
\end{remark}
\begin{proof}
  Subtracting~\eqref{eq:semi_weak_stress} from~\eqref{eq:full_stress}, we see that
  \begin{equation*}
    \iprod{\Dtt \euh^n}{\Chi} + \iprod{\bA \strain{\Dt\euh^n} + \bB \strain{\euh^n}}{\strain{\Chi}} =  \iprod{\bM \ethh^n}{\strain{\Chi}} 
  \end{equation*}
for all $\Chi \in \Sdh$. Now let $\euh^n = \eta^n + \rho^n$, where
\begin{equation*}
  \eta^n = U_h^n - W^n \in \Sdh \quad \text{and} \quad \rho^n = W^n - U^n,
\end{equation*}
with the discrete Ritz-Volterra projection $W^n$ of $U^n$ satisfying $W^0 = U^0 = 0$ and
\begin{equation} \label{eq:RitzVolterra}
  \iprod{\bA \strain{\Dt W^n - \Dt U^n} + \bB \strain{W^n - U^n}}{\strain{\Chi}} = 0
\end{equation}
for all $\Chi \in \Sdh$. We note that Equation~\eqref{eq:RitzVolterra} may also be stated as 
\begin{equation*}
  \iprod{\bA \strain{\Dt \rho^n} + \bB \strain{\rho^n}}{\strain{\Chi}} = 0,
\end{equation*}
and that since $\Dt U^0 = 0$, also $\Dt W^0 = 0$. Additionally, we need the Ritz projection $\bR_h$ given by the viscosity term. For a generic $u \in \Vd$, this is defined by
\begin{equation*}
\iprod{\bA \strain{\bR_h u - u}}{\strain{\Chi}} = 0
\end{equation*}
for all $\Chi \in \Sdh$, and we have the inequality
\begin{equation*}
  \norm{\bR_h u - u} + h \norm{\bR_h u -u}_{\Vd} \le Ch^2 \norm{u}_{\bH^2}.
\end{equation*}

We start by estimating the $\Vd$-norms of $\Dt \rho^n$ and $\rho^n$. To this end, we observe that for a generic $u$, we have
\begin{equation*}
  \norm{u}_{\Vd}^2 = \norm{\strain{u}}_Q^2 \le \norm{\nabla u}_Q^2 = \sum_{j=1}^{d}{\Bignorm{\frac{\partial u}{\partial x_j}}^2}
\end{equation*}
and that
\begin{equation*}
  \Bignorm{\frac{\partial u}{\partial x_j}} = \sup_{\varphi \in C_0^{\infty}(\Omega)^d, \norm{\varphi} = 1} \iprod{\frac{\partial u}{\partial x_j}}{\varphi} .
\end{equation*}
We therefore take $\varphi \in C_0^{\infty}(\Omega)^d$ with $\norm{\varphi} = 1$ and let $\Psi \in \Vd$ be the solution to 
\begin{equation*}
  \iprod{\bA \strain{\Psi}}{\strain{\Chi}}_Q = - \iprod{\frac{\partial \varphi}{\partial x_j}}{\Chi}.
\end{equation*}
Then 
\begin{align*}
  \iprod{\frac{\partial \Dt \rho^n}{\partial x_j}}{\varphi} &= - \iprod{\Dt \rho^n}{\frac{\partial \varphi}{\partial x_j}} = \iprod{\bA \strain{\Psi}}{\strain{\Dt \rho^n}} = \iprod{\bA \strain{\Dt \rho^n}}{\strain{\Psi}} \\
&= \iprod{\bA \strain{\Dt \rho^n}}{\strain{\Psi - \bR_h \Psi}} + \iprod{\bA \strain{\Dt \rho^n}}{\strain{\bR_h \Psi}} \\
&= \iprod{\bA \strain{\Dt \rho^n}}{\strain{\Psi - \bR_h \Psi}} - \iprod{\bB \strain{\rho^n}}{\strain{\bR_h \Psi}} =: R_1 + R_2,
\end{align*}
where the last term is bounded by
\begin{equation*}
  R_2 \le C \norm{\rho^n}_{\Vd} \norm{\bR_h \Psi}_{\Vd} \le C \norm{\rho^n}_{\Vd} (\norm{\bR_h \Psi - \Psi}_{\Vd} + \norm{\Psi}_{\Vd}) \le C \norm{\rho^n}_{\Vd}.
\end{equation*}
Moreover, since $\Dt W^n \in \Sdh$, the first term is bounded by
\begin{align*}
  R_1 = -\iprod{\bA \strain{\Dt U^n}}{\strain{\Psi - \bR_h \Psi}} &= \iprod{\bA \strain{\bR_h \Dt U^n - \Dt U^n}}{\strain{\Psi - \bR_h \Psi}} \\
&= \iprod{\bA \strain{\bR_h \Dt U^n - \Dt U^n}}{\strain{\Psi}} \\
&\le C\norm{\bR_h \Dt U^n - \Dt U^n}_ {\Vd} \norm{\Psi}_{\Vd} \\
&\le Ch \norm{\Dt U^n}_{\bH^2}.
\end{align*}
By expressing $\rho^n$ in terms of $\Dt \rho^j$ and noting that $\rho^0 = 0$, we thus have
\begin{equation*}
  \norm{\Dt \rho^n}_ {\Vd} \le Ch \norm{\Dt U^n}_{\bH^2} + Ck \sum_{j=1}^{n}{\norm{\Dt \rho^j}_{\Vd}},
\end{equation*}
and under the step size restriction $Ck < 1$ we can eliminate the last term of the sum and apply Grönwall's lemma. This shows that
\begin{equation*} 
  \norm{\Dt \rho^n}_{\Vd} \le Ch \Big( \norm{\Dt U^n}_{\bH^2}  + Ck \sum_{j=1}^{n-1}{\norm{\Dt U^j}_{\bH^2}} \Big) .
\end{equation*}
By using the regularity shown in Theorem~\ref{theorem:time_error} and then summing over $n$, we see that
\begin{equation*}
  \norm{\rho^n}_{\Vd} + \norm{\Dt \rho^n}_{\Vd} \le Ch.
\end{equation*}
Using these bounds we may now estimate $\rho$ also in the $\bL^2$-norm, by instead letting $\Psi \in \Vd$ be the solution to 
\begin{equation*}
  \iprod{\bA \strain{\Psi}}{\strain{\Chi}}_Q = - \iprod{\varphi}{\Chi}.
\end{equation*}
Then as before,
\begin{equation*}
  \iprod{\Dt \rho^n}{\varphi} = \iprod{\bA \strain{\bR_h \Dt U^n - \Dt U^n}}{\strain{\Psi}} + \iprod{\bB \strain{\rho^n}}{\strain{\bR_h \Psi}} =: R_3 + R_4,
\end{equation*}
where
\begin{equation*}
  R_3 \le C\norm{\bR_h \Dt U^n - \Dt U^n}_ {\Vd} \norm{\Psi}_{\Vd} \le Ch^2 \norm{\Dt U^n}_{\bH^2}.
\end{equation*}
For $R_4$, we note that $\norm{\Psi}_{\bH^2} \le C \norm{\varphi} \le C$, so that by using integration by parts and observing that both $\rho^n$ and $\Psi$ are zero on $\partial \Omega$ we get,
\begin{align*}
  R_4 &\le \iprod{\bB \strain{\rho^n}}{\strain{\bR_h \Psi - \Psi}} + \iprod{\bB \strain{\rho^n}}{\strain{\Psi}}\\
&\le C\norm{\rho^n}_{\Vd}\norm{\bR_h\Psi - \Psi}_{\Vd} + C \norm{\rho^n}\norm{\Psi}_{\bH^2} + \norm{\rho^n}_{\bL^2(\partial\Omega)}\norm{\Psi}_{\bH^1(\partial\Omega)} \\
&\le Ch^2 + C\norm{\rho^n}.
\end{align*}
Hence similarly to the calculation for the $\Vd$-norm, Grönwall's lemma implies that
\begin{equation*} 
  \norm{\Dt \rho^n} \le Ch^2 \Big( \norm{\Dt U^n}_{\bH^2}  + Ck \sum_{j=1}^{n-1}{\norm{\Dt U^j}_{\bH^2}} \Big) ,
\end{equation*}
so that
\begin{equation*}
  \norm{\rho^n} + \norm{\Dt \rho^n} \le Ch^2.
\end{equation*}
To bound $\eta^n$, we also need a bound on the second derivative of $\rho^n$. For this, we apply $\Dt$ to~\eqref{eq:RitzVolterra} and then follow the same procedure as above. This shows that 
\begin{equation*}
  \norm{\Dtt \rho^n}_{\Vd} \le Ch \Big( \norm{\Dtt U^n}_{\bH^2}  + Ck \sum_{j=1}^{n-1}{\norm{\Dtt U^j}_{\bH^2}} \Big),
\end{equation*}
and similarly for the $\bL^2$-norm, but with $h^2$ instead of $h$. 
We do not have pointwise $\bH^2$-regularity of $\Dtt U^n$ from Theorem~\ref{theorem:time_error}, but we may estimate the sum by 
\begin{equation*}
k \sum_{j=1}^{n-1}{\norm{\Dtt U^j}_{\bH^2}} \le  \Big( k \sum_{j=1}^{n-1}{\norm{\Dtt U^j}_{\bH^2}^2}\Big)^{1/2} \le C,
\end{equation*}
and conclude that
\begin{equation} \label{eq:Dtt_rho_Vbound}
  \norm{\Dtt \rho^n} + h\norm{\Dtt \rho^n}_{\Vd} \le Ch^2 + Ch^2 \norm{\Dtt U^n}_{\bH^2}.
\end{equation}
Here the $\norm{\Dtt U^n}_{\bH^2}$-term is not necessarily finite, but since this bound will only be used inside a sum it causes no problems.

Now for $\eta^n$, by using~\eqref{eq:RitzVolterra} to exchange $W^n$ for $U^n$ and then~\eqref{eq:semi_weak_stress}, \eqref{eq:full_stress}, we get
\begin{align*}
  \iprod{\Dtt \eta^n}{\Chi} &+ \iprod{\bA \strain{\Dt \eta^n} + \bB \strain{\eta^n}}{\strain{\Chi}} \\
                            &= \iprod{\Dtt U^n - \Dtt W^n}{\Chi} + \iprod{\bM\ethh^n}{\strain{\Chi}}  \\
&= -\iprod{\Dtt \rho^n}{\Chi} + \iprod{M\ethh^n}{\strain{\Chi}} 
\end{align*}
Choosing $\Chi = \Dt \eta^n \in \Sdh$, by~\eqref{eq:Dtt_rho_Vbound} we get, after canceling a $C_2 \norm{\Dt \eta^n}_{\Vd}^2$ term,
\begin{equation*}
  \Dt \norm{\Dt \eta^n}^2 + C_2 \norm{\Dt \eta^n}_{\Vd}^2 + \Dt \norm{\eta^n}_{\bB}^2 \le  C\big(h^4 + h^4 \norm{\Dtt U^n}_{\bH^2}^2 + \norm{\ethh^n}^2\big),
\end{equation*}
so summing and noting again that $k \sum_{j=1}^{n-1}{\norm{\Dt U^j}_{\bH^2}^2} \le C$, we have
\begin{equation*}
  \norm{\Dt \eta^n}^2 + k \sum_{j=1}^{n-1}{ \norm{\Dt \eta^j}_{\Vd}^2} + \norm{\eta^n}_{\Vd}^2 \le Ch^4 + Ck \sum_{j=1}^{n-1}{\norm{\ethh^j}^2}.
\end{equation*}
Finally, combining the bounds for $\rho^n$, $\eta^n$ and their first derivatives leads to the statement of the lemma. 
\end{proof}
\begin{remark}
  We note that the regularity given in Theorem~\ref{theorem:time_error} is not enough to show $\norm{\Dt \euh^n}_{\Vd}^2 \le Ch^2 + Ck \sum_{j=1}^{n}{\norm{\ethh^j}^2}$, but such a bound is not required for the proof of the next theorem. 
\end{remark}

\begin{theorem} \label{theorem:full_error}
  Let Assumptions~\ref{ass:viscoelasticity}-\ref{ass:regularity} be satisfied and $(\Theta^n, \Phi^n, U^n)$ and $(\Theta_h^n, \Phi_h^n, U_h^n)$ be solutions to equations~\eqref{eq:semi_weak_T}--\eqref{eq:semi_weak_stress} and~\eqref{eq:full_T}--\eqref{eq:full_stress}, respectively. Then there are positive constants $k_0$ and $h_0$ such that if $k < k_0$ and $h < h_0$ then for $n = 1, \ldots, N$,
  \begin{equation*}
    \norm{\ethh^n} +  \norm{\ephih^n} + \norm{\Dt \euh^n} \le Ch^2 \quad \text{and} \quad  \norm{\ethh^n}_{H^1} + \norm{\ephih^n}_{H^1} + \norm{\Dt \euh^n}_{\Vd} \le Ch,
  \end{equation*}
with the constant $C$ independent of $k$, $h$ and $n$.
\end{theorem}

\begin{proof}
The idea is, similarly to the time-discrete case, essentially to write down the equation for $\ethh^n$, test it with $\ethh^n$, express the errors $\euh^n$ and $\ephih^n$ in terms of $\ethh^j$ by Lemma~\ref{lemma:RitzVolterra} and its potential-analogue, and finally use Grönwall's lemma. However, since $\ethh^n$ does not belong to the finite element space, we need to introduce instead
\begin{equation*}
  \eh^n = \Theta_h^n - R_h \Theta^n,
\end{equation*}
where $R_h$ denotes the Ritz projection onto $\Sh$% (rather than $\bR_h : \Vd \to\Sdh$)
. Due to Theorem~\ref{theorem:time_error} we then have $\norm{\ethh^n} \le \norm{\eh^n} + \norm{R_h \Theta^n - \Theta^n} \le \norm{\eh^n} + Ch^2$. It follows that for all $\chi \in \Sh$,
\begin{equation*}
  \iprod{\Dt \eh^n}{\chi} + \iprod{\nabla \theta_h^n}{\nabla \chi} = \iprod{\Dt (\Theta^n - R_h \Theta^n)}{\chi} + \iprod{R_\phi}{\chi} - \iprod{M : \strain{\Dt \euh^{n-1}}}{\chi},
\end{equation*}
where $R_\phi$ contains terms related to the potential $\phi$. Choosing $\chi = \eh^n$, we know from~\cite{LiSun2013} that 
\begin{equation*}
  \iprod{R_\phi}{\eh^n} \le Ch^3 + Ch^4\norm{\Dt \Theta^n}_{H^2}^2 + Ch^{-1}\norm{\eh^{n-1}}^4 + C\norm{\eh^{n-1}}^2 + \frac{1}{4} \norm{\eh^n}_{H^1}^2,
\end{equation*}
and we also have by~\eqref{eq:strain_by_parts} that
\begin{equation*}
  \iprod{M : \strain{\Dt \euh^{n-1}}}{\eh^n} \le C\norm{\Dt \euh^{n-1}}^2 + \frac{1}{4} \norm{\eh^n}_{H^1}^2.
\end{equation*}
We additionally know that $\norm{\eh^0} = \norm{I_h \theta_0 - \theta_0} \le Ch^2 < h^{1/2}$ if $h < h_0$. Assuming that $\norm{\eh^{m}} \le h^{1/2}$ for $m = 1, \ldots, n-1$ therefore means that
\begin{equation*}
  \Dt \norm{\eh^m}^2 + \norm{\eh^m}_{H^1}^2 \le Ch^3 + Ch^4\norm{\Dt \Theta^m}_{H^2}^2 + C\norm{\eh^{m-1}}^2 + C\norm{\Dt \euh^{m-1}}^2
\end{equation*}
for $m = 1, \ldots, n$, which after summation and usage of Lemma~\ref{lemma:RitzVolterra} yields
\begin{align*}
  \norm{\eh^m}^2 + k\sum_{j=1}^{m}{\norm{\eh^j}_{H^1}^2} &\le Ch^3 + Ch^4 + Ck\sum_{j=1}^{m-1}{\norm{\eh^{j}}^2} + Ck\sum_{j=1}^{m-1}{\norm{\Dt \euh^{j}}^2} \\
&\le Ch^3 + Ck\sum_{j=1}^{m-1}{\Big( \norm{\eh^{j}}^2 + Ck\sum_{i=1}^{j}{\norm{\eh^{i}}^2}\Big)}.
\end{align*}
If we now set $g_m = \max_{1 \le j \le m}\big( \norm{\eh^{j}}^2 + Ck\sum_{i=1}^{j}{\norm{\eh^{i}}^2}\big)$ we have
\begin{equation*}
  g_m \le Ch^3 + Ck \sum_{j=1}^{m-1}{g_j},
\end{equation*}
to which we may apply Grönwall's lemma to acquire
\begin{equation*}
  \norm{\eh^{n}}^2 + Ck\sum_{j=1}^{n}{\norm{\eh^{j}}^2} \le \tilde{C}h^3.
\end{equation*}
Hence if $\tilde{C}h^{5/2} \le 1$ we have that $\norm{\eh^n} \le h^{1/2}$. Thus by induction $\norm{\eh^n} \le h^{1/2}$ holds for all $n$ such that $0\le n \le N$. But then also the other calculations just performed are valid for $1\le n \le N$, so in fact $\norm{\eh^n} \le h^{3/2}$. This preliminary bound may be used as in~\cite[p.\ 631]{LiSun2013} to show $\norm{\ephih^n} \le Ch$ and to improve the bound of the quadratic potential term to
\begin{equation*}
  \iprod{R_\phi}{\eh^n} \le Ch^4 + Ch^4\norm{\Dt \Theta^n}_{H^2}^2 + C\norm{\eh^{n-1}}^2 + \frac{1}{4} \norm{\eh^n}_{H^1}^2.
\end{equation*}
Hence,
\begin{align*}
  \norm{\eh^n}^2 + k\sum_{j=1}^{n}{\norm{\eh^j}_{H^1}^2} \le Ch^4 + Ck\sum_{j=1}^{m-1}{\Big(\norm{\eh^{j}}^2 + Ck\sum_{i=1}^{j}{\norm{\eh^{i}}^2}\Big)},
\end{align*}
and once more applying Grönwall's lemma to $g_n$ shows that
\begin{equation*}
  \norm{\eh^n}^2 + k\sum_{j=1}^{n}{\norm{\eh^j}_{H^1}^2} \le Ch^4.
\end{equation*}
This proves $\norm{\ethh^n} \le Ch^2$, and from~\cite{LiSun2013} we find ${\norm{\ephih^n} + h\norm{\ephih^n}_{H^1} \le Ch^2}$. Applying Lemma~\ref{lemma:RitzVolterra} gives $\norm{\Dt \euh^n} \le Ch^2$. Finally, by inverse inequalities we find also that $\norm{\ethh^n}_{H^1} + \norm{\Dt \euh^n}_{\Vd} \le Ch$.
\end{proof}

\begin{proof}[Proof (of Theorem~\ref{theorem:main})]
This follows directly from Theorem~\ref{theorem:time_error} and Theorem~\ref{theorem:full_error} upon observing that, e.g.,
\begin{equation*}
  \norm{\Dt U_h^n - \dot{u}_n} \le \norm{\euh} + \norm{\eu} + \norm{\Dt u_n - \dot{u}_n},
\end{equation*}
 where the last term is bounded in the proper way due to the regularity assumptions on the solution to the continuous system.
\end{proof}

\section{Numerical experiments} \label{sec:experiments}
We have implemented both the method based on~\eqref{eq:full_T}--\eqref{eq:full_stress} and the corresponding fully implicit method based on implicit Euler, using FEniCS (see e.g.~\cite{AlnaesBlechtaEtAl2015,LoggMardalEtAl2012a}). These implementations were then used to verify our theoretical results by applying them to the following test examples.

\subsection{Problem  1} \label{subsec:problem1}
First consider the two-dimensional problem with $\Omega = (0, 1)^2$, $\bM = I$, $f = [0,0]^T$ and the viscosity and elasticity tensors given in Voigt notation by
\begin{equation*}
  \bA = \bB =
  \begin{bmatrix}
    1 & 1 & 0 \\
    1 & 1 & 0 \\
    0 & 0 & 1 
  \end{bmatrix} .
\end{equation*}
We take the electrical conductivity to be given by 
\begin{equation*}
   \sigma(\theta) = 2.5 - \arctan(5\theta - 10),
\end{equation*}
which has a rather steep slope close to $\theta = 2$. The initial conditions are given by $\theta_0(x, y) = 0$ and $u_0(x,y) = v_0(x,y) = [0, 0]^T$. These functions also define the Dirichlet boundary conditions for $\theta$ and $u$, while for $\phi$ they are given by $\phi_b(x,y) = 5(1-x)$.

We discretize $\Omega$ by first subdividing it into squares and then dividing each square into four triangles. With $N_x$ squares in each dimension, each triangle has diameter $h = 1/N_x$ and the full grid has $4N_x^2$ triangles. We take $N_x \in \{4, 8, 16, 32, 64\}$. 
Since the error should be $\ordo(h^2 + k)$, we choose the number of time steps to be $N_t = N_x^2 / 2$. With the final time $\tend = 1$, this gives $k = 2 h^2$. We emphasize here that the time steps could be taken much larger than this, but illustrating the error is then less straightforward. Finally, because the exact solution of the problem is not available we cannot compute the exact errors. Instead, we compare the different approximations to a reference approximation $(\Theta_{\text{ref}}, \Phi_{\text{ref}}, U_{\text{ref}})$ computed by the implicit Euler scheme with $N_x = 128$ and $N_t = 8192$.

Figure~\ref{fig:problem1} shows the errors
\begin{equation} \label{eq:errors}
  \max_{1 \le n \le N_t} \norm{\Theta_h^n - \Theta_{\text{ref}}(t_n)}_{L^2}, 
\max_{1 \le n \le N_t} \norm{\Phi_h^n - \Phi_{\text{ref}}(t_n)}_{L^2} \quad \text{and} \quad \max_{1 \le n \le N_t} \norm{U_h^n - U_{\text{ref}}(t_n)}_{\bL^2}
\end{equation}
for the different discretizations on a logarithmic scale, for both the semi-implicit method (left) and the method based on implicit Euler (right). These clearly exhibit the expected error behaviour predicted by Theorem~\ref{theorem:full_error}, except for the first points where the grid is very coarse. We also note that the errors are very similar in size, which means that the semi-implicit method is much more efficient. A peculiar effect in this case is that the semi-implicit errors in $\theta$ and $\phi$ are actually less than the implicit Euler errors, though this does not hold for the error in $u$.

\begin{figure}[t!]
\centering\includegraphics[width=0.49\columnwidth]{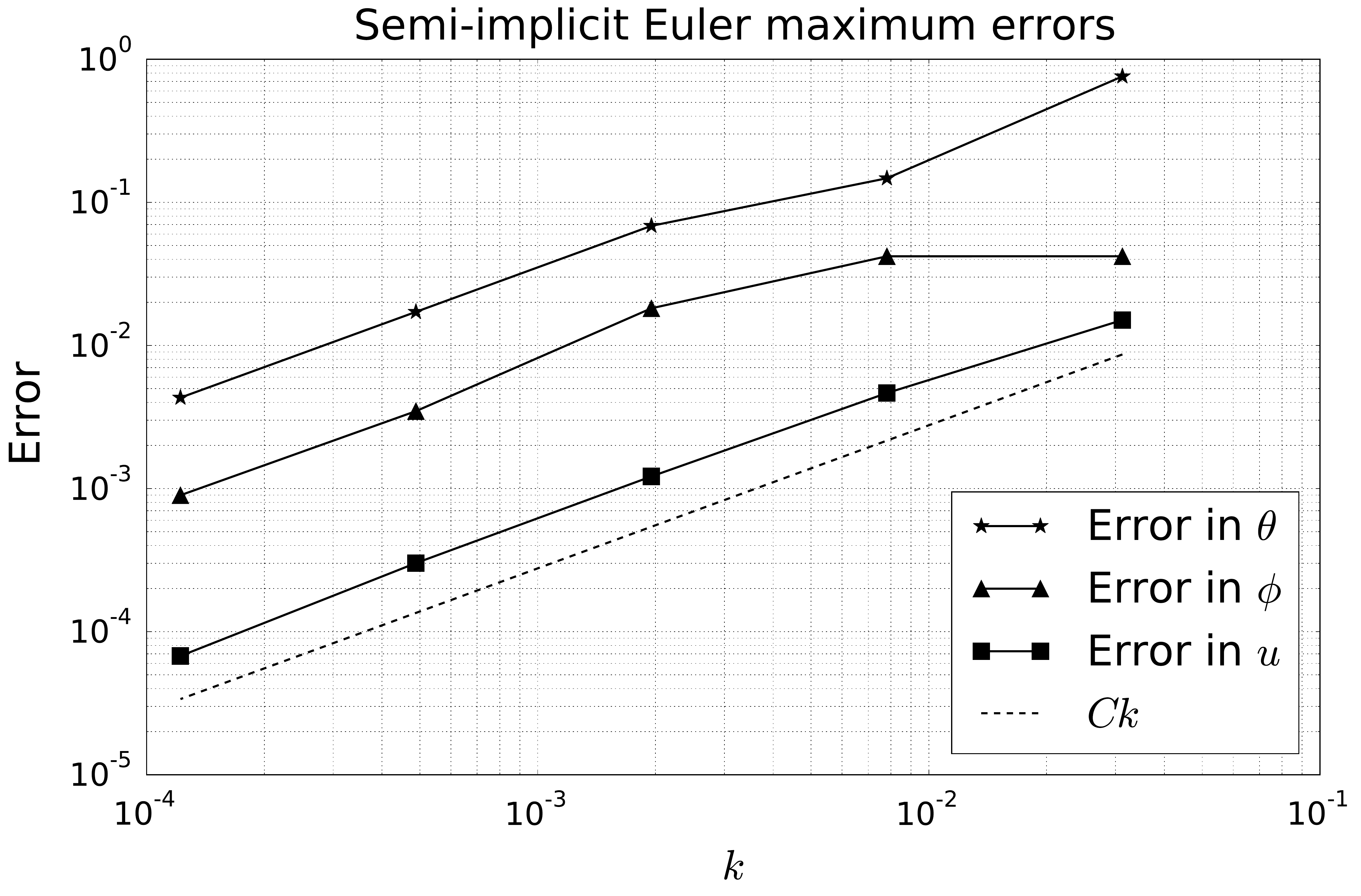}
\includegraphics[width=0.49\columnwidth]{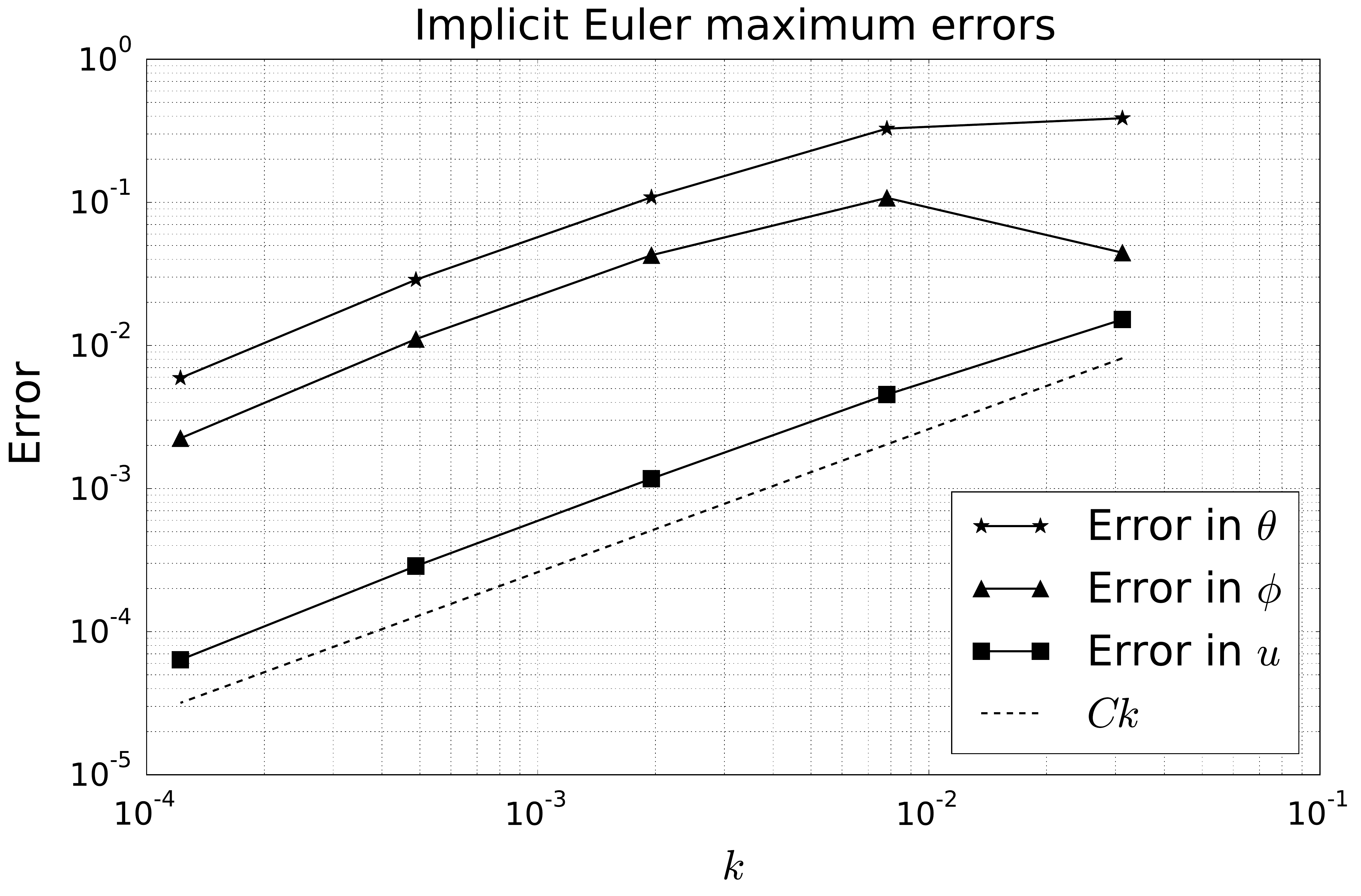}
\caption{The errors~\eqref{eq:errors} for the problem defined in Section~\ref{subsec:problem1}, computed by the semi-implicit method (left) and the implicit Euler method (right).}
\label{fig:problem1}
\end{figure}

\subsection{Problem  2} \label{subsec:problem2}
In the second experiment, we investigated the influence of the viscosity on the errors. To this end, we employ the same data as presented in Section~\ref{subsec:problem1} except for the viscosity operator which we set to
\begin{equation*}
  \bA = \gamma
  \begin{bmatrix}
    1 & 1 & 0 \\
    1 & 1 & 0 \\
    0 & 0 & 1 
  \end{bmatrix} 
\end{equation*}
 (in Voigt notation).
In this case, we used $N_x \in \{4, 8, 16, 32\}$ with $N_t = N_x^2 / 4 $ and took $N_x = 64$, $N_t = 1024$ for the reference approximation. We only used the semi-implicit scheme here. The first observation is that varying $\gamma$ has essentially no effect on the errors in $\theta$ and $\phi$. This is to be expected, as the influence of $u$ on $\theta$ is not so large. We therefore omit the plots of these errors, and instead present the error in $u$ for different values of $\gamma$ in Figure~\ref{fig:problem2}.

We observe that the error clearly increases as $\gamma$ is decreased, which is to be expected. Indeed, an inspection of the convergence proof indicates that the $L^2$-error should be inversely proportional to the coercivity constant of $\bA$, and thus also of $\gamma$. This is, however, in the worst case. In the current situation, Figure~\ref{fig:problem2} indicates that even $\gamma = 0$ would be perfectly feasible, though smaller step sizes might be necessary to enter the asymptotic regime.

\begin{figure}[t!]
\centering\includegraphics[width=0.60\columnwidth]{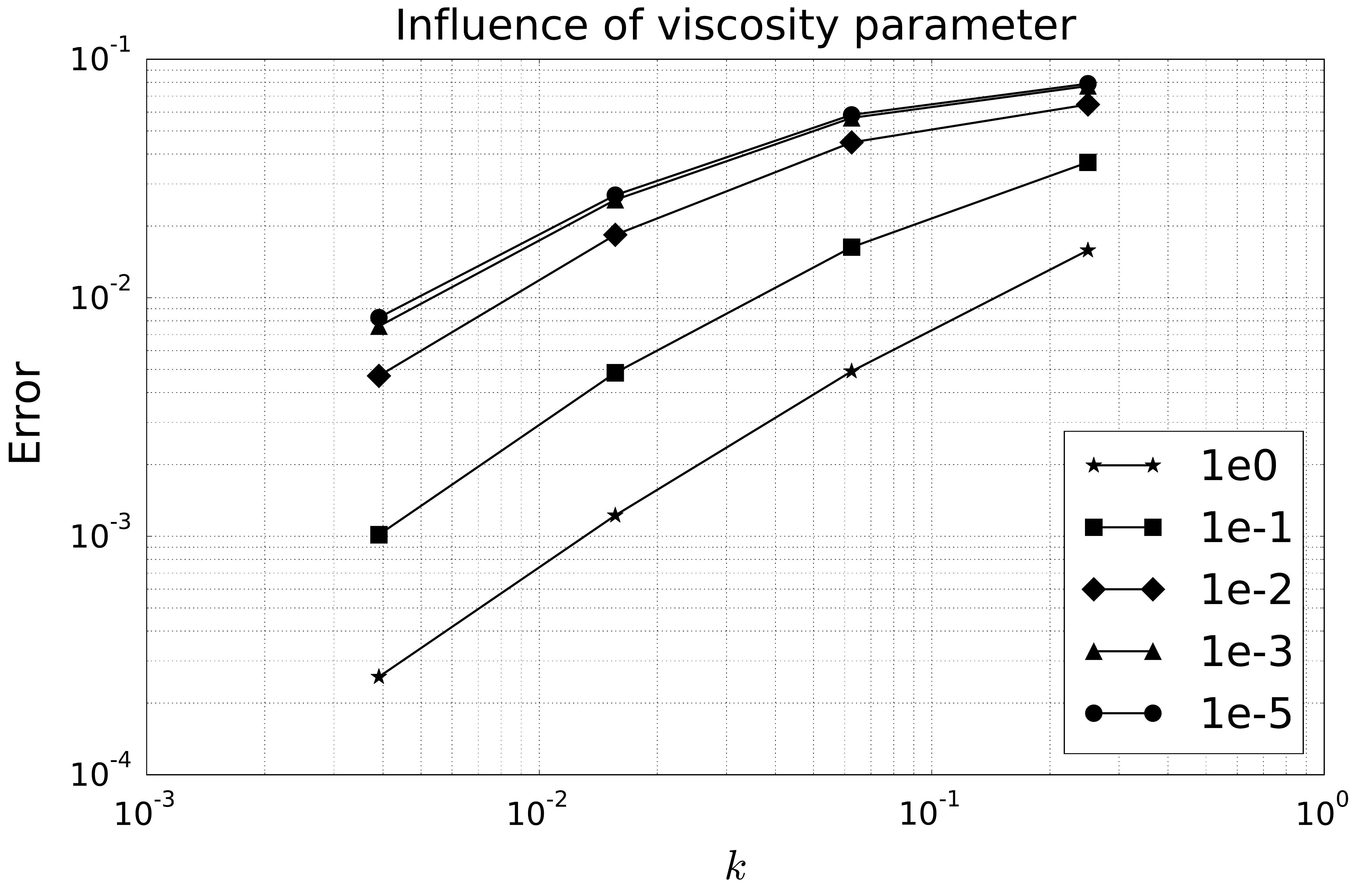}
\caption{The errors $\max_{1 \le n \le N_t} \norm{U_h^n - U_{\text{ref}}^n}_{\bL^2}$ for the problem defined in Section~\ref{subsec:problem2}, computed by the semi-implicit method. The different curves correspond to the different values of $\gamma \in \{10^0, 10^{-1}, 10^{-2}, 10^{-3}, 10^{-5} \}$.}
\label{fig:problem2}
\end{figure}

\subsection{Problem 3} \label{subsec:problem3}
For our last numerical experiment, we consider a 3D problem arising from an engineering application, inspired by~\cite{Henneken_etal2006} and~\cite{Holst_etal2010}. We let $\Omega$ be as in Figure~\ref{fig:thermistor_geometry}, which also shows a typical spatial tetrahedral discretization. This represents a micro-electro{-}mechanical system (MEMS) used for precise positioning on small scales. When an electric current is passed through the device from the upper-left connector to the lower-left connector, it heats up. This causes a deformation, which due to the asymmetrical design of the component makes the tip move downwards.

\begin{figure}[t!]
\centering\includegraphics[width=0.60\columnwidth]{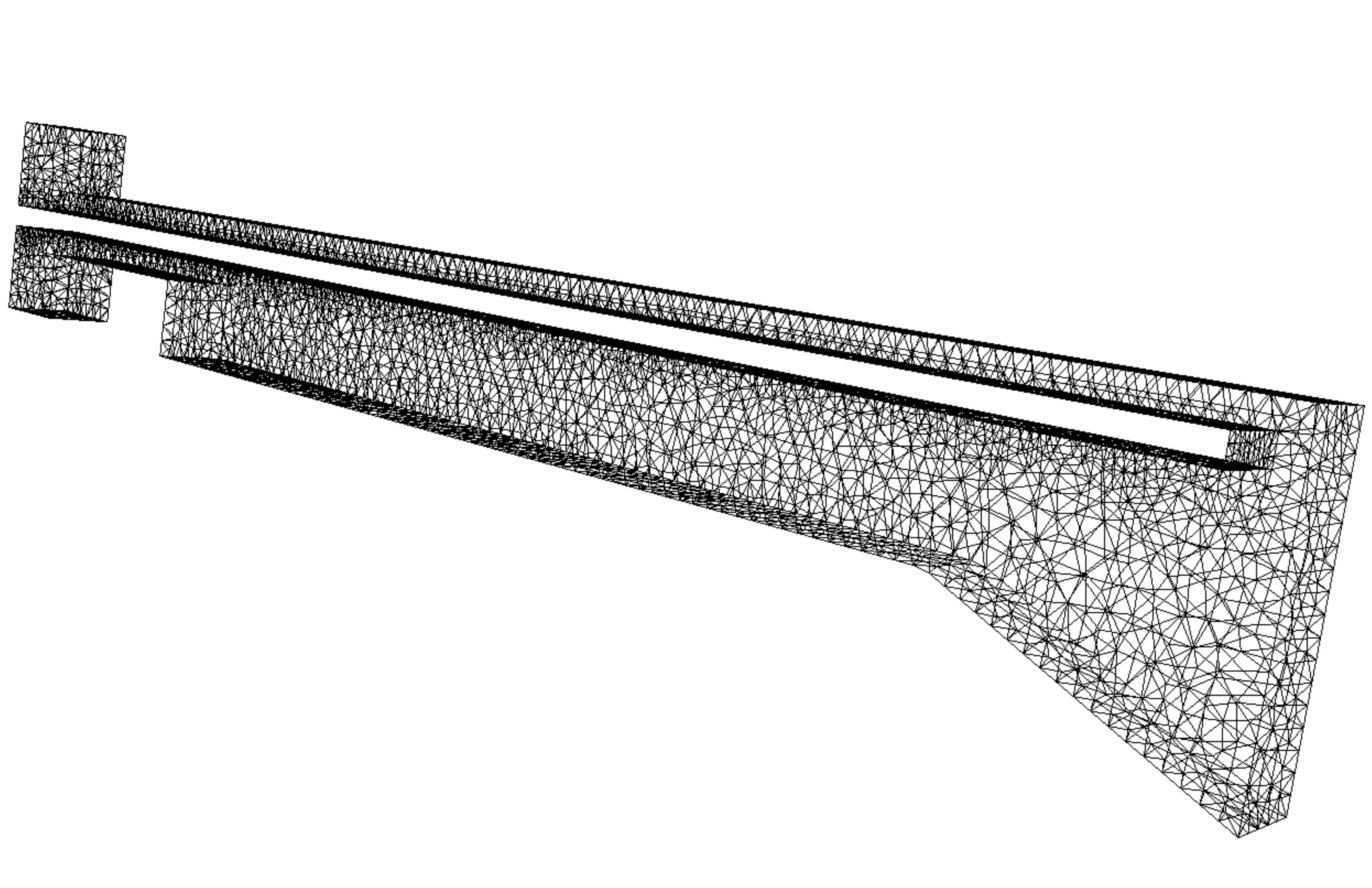}
\caption{A mesh for the problem described in Section~\ref{subsec:problem3}. The outer dimensions are $192 \times 27 \times \SI{9}{\micro \metre}$.}
\label{fig:thermistor_geometry}
\end{figure}

We employ homogeneous Neumann boundary conditions everywhere except for at the left-most edge of the two connectors. These correspond to the component being insulated and stress-free. On the left-most edge we choose the Dirichlet boundary conditions
\begin{equation*}
  \theta = 0, \quad \phi =
  \begin{cases}
    50, \quad z > 0 \\
    \hphantom{5}0, \quad z < 0
  \end{cases}, \quad \text{and} \quad 
u = v =
\begin{bmatrix}
  0\\
  0
\end{bmatrix},
\end{equation*}
corresponding to the component being clamped and having a potential difference applied between the two connectors. The equations, including physical constants, are
\begin{align}
  \label{eq:mems_T}
 \rho c \dot{\theta} &= \nabla \Big( \bK \nabla \theta \Big) + \sigma(\theta) | \nabla \phi|^2 -  \Theta_0 \bM : \strain{\dot{u}}, \\
  \label{eq:mems_phi}
   0 &= \nabla \cdot \big( \sigma(\theta) \nabla \phi \big), \\
  \label{eq:mems_stress}
\rho \ddot{u} &= \nabla \cdot \big( \bA\strain{\dot{u}} + \bB\strain{u} - \bM \theta \big) + f. 
\end{align}
Here, $\rho$ denotes the density, $c$ the specific heat capacity, $\bK = k\bI$ the thermal conductivity matrix,  $\bM = m\bI$ the thermal expansion matrix and $\sigma$ the electrical conductivity. 
Additionally, $\theta$ indicates the deviation from the ambient temperature $\Theta_0 = \SI{293.15}{\kelvin}$.

We choose the elasticity and viscosity operators to be given on Lam\'{e} parameter form:
\begin{equation*}
  \bA \strain{\dot{u}} = 2\eta_1 \strain{\dot{u}} + \eta_2 \trace \strain{\dot{u}}\bI
 \quad \text{and } \quad 
  \bB \strain{u} = 2\mu \strain{u} + \lambda \trace \strain{u} \bI,
\end{equation*}
where
\begin{equation*}
  \mu = \frac{E}{2(1+\nu)} \quad \text{and } \quad \lambda = \frac{E\nu}{(1+\nu)(1-2\nu)}
\end{equation*}
are given in terms of Poisson's ratio $\nu$ and Young's modulus $E$, and $\eta_1$, $\eta_2$ are corresponding viscosity parameters. Here, $\trace$ denotes the trace of a matrix; $\trace \tau = \tau_{11} + \tau_{22}$.

The parameter values we have used, similar to the material properties of silicon, are listed in Table~\ref{table:mems_parameters}. In addition to this, we take $f = [0,0,0]^T$ and choose the electrical conductivity as
\begin{equation*}
  \sigma(\theta) = \frac{\num{38e6}}{27} \bigg( 3000 + 550\Big(\frac{\pi}{2} + \arctan\frac{\theta_1 - 250}{250}\Big) \bigg)^{-1} \si{\siemens \per \meter},
\end{equation*}
where $\theta_1 = \Theta_0 + \theta$.

\begin{table}
  \centering
  \begin{tabular}{crc | crc}
    Parameter & Value & Unit & Parameter & Value & Unit \\
    \hline
    $\rho$      & \num{2.329e3} & \si{\kg \per \cubic \meter } & 
    $c$         & \num{0.7046e3}  & \si{\joule \per \kg \per \kelvin} \\ 
    $k$         & \num{158} & \si{\watt \per \meter \per \kelvin} &
    $m$         & \num{1.3265e5} & \si{\newton \per \square \meter \per \kelvin} \\
    $\nu$  & \num{0.01} & \si{1} &
    $E$  & \num{150e7}& \si{\newton \per \square \meter} \\
    $\eta_1$  & \num{1e6} & \si{\newton \second \per \square \meter} &
    $\eta_2$  & \num{5e6} & \si{\newton \second \per \square \meter} \\
  \end{tabular}
  \caption{Parameter values utilized in Problem 3.}
  \label{table:mems_parameters}
\end{table}

We solve the problem until the time $\tend = 0.1$ using the semi-implicit method for different spatial and temporal discretizations. The maximum sizes $h$ of the tetrahedrons that were used and the corresponding number of vertices are listed in Table~\ref{table:problem3_errors}. The time steps were again taken proportional to $h^2$ but modified slightly to yield an integer number of steps. Since the temporal grids thus generated are not refinements of each other, we measured the error as the sum of the errors at only the points $t_j = j\cdot 10^{-2}$ for $j = 1, \ldots, 10$. These errors are listed in Table~\ref{table:problem3_errors}, and also plotted in Figure~\ref{fig:MEMS_errors}. While we cannot apply Theorem~\ref{theorem:full_error} directly, due to the mixed boundary conditions and the non-convexity of the domain, we observe that we still acquire almost $\ordo(h^2+k)$ convergence. The curves wiggle because $k = Ch^2$ is only approximately satisfied, and the different magnitudes of the errors reflect the relative sizes of the solution components. The larger error in $\theta$ for the coarsest mesh indicates that it violates either the $k < k_0$ or $h < h_0$ mesh size limitations.
\begin{table}
  \centering
  \begin{tabular}{S|S|S|S|S|S}
    h & k & {Vertices} & {Error in $\theta$} & {Error in $\phi$} & {Error in $u$}  \\
\hline
4.82e-06 & 5.000e-3 &  5219 & 1.4445e-01 &  1.4212e-01 &  9.6544e-01 \\
3.56e-06 & 3.333e-3 &  7510 & 2.7433e-03 &  2.1775e-03 &  1.9998e-01 \\
2.80e-06 & 2.000e-3 & 11783 & 1.6006e-03 &  1.2654e-03 &  1.2388e-01 \\
2.39e-06 & 1.667e-3 & 18719 & 1.2219e-03 &  9.5176e-04 &  8.8905e-02 \\
2.01e-06 & 1.111e-3 & 28473 & 7.7211e-04 &  6.0036e-04 &  5.0397e-02 \\
1.33e-06 & 5.263e-4 & 85310 & {-} & {-} & {-}
  \end{tabular} 
  \caption{Spatial and temporal discretizations parameters as well as maximal errors for the MEMS problem (Section~\ref{subsec:problem3}) at the time points $t_j = j\cdot 10^{-2}$ for $j = 1, \ldots, 10$. The last line corresponds to the reference approximation. }
  \label{table:problem3_errors}
\end{table}

\begin{figure}[t!]
\centering
\includegraphics[width=0.6\columnwidth]{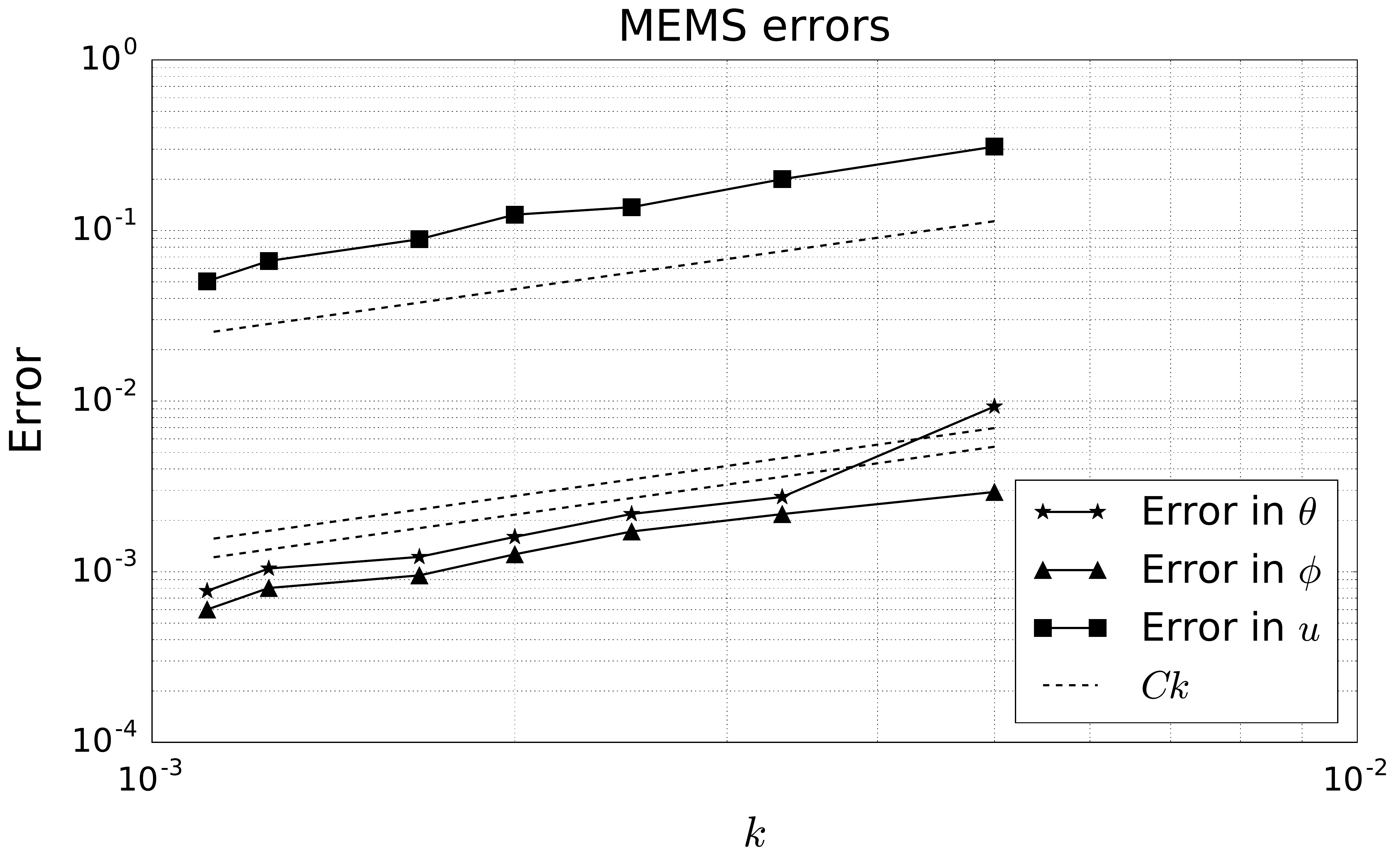}
\caption{Maximal errors at the time points $t_j = j\cdot 10^{-2}$ for $j = 1, \ldots, 10$ for the MEMS problem defined in Section~\ref{subsec:problem3}. The lines wiggle because $k = Ch^2$ is only approximately satisfied.}
\label{fig:MEMS_errors}
\end{figure}

Finally, Figure~\ref{fig:MEMS_solutions} shows the approximations $\Theta_h^N$, $\Phi_h^N$ and $U_h^N$ at $\tend$, viewed from the side. At this point in time the solutions have just reached their steady state, and we see that the body deforms in the expected fashion.

\begin{figure}[t!]
\centering
\includegraphics[width=0.32\columnwidth]{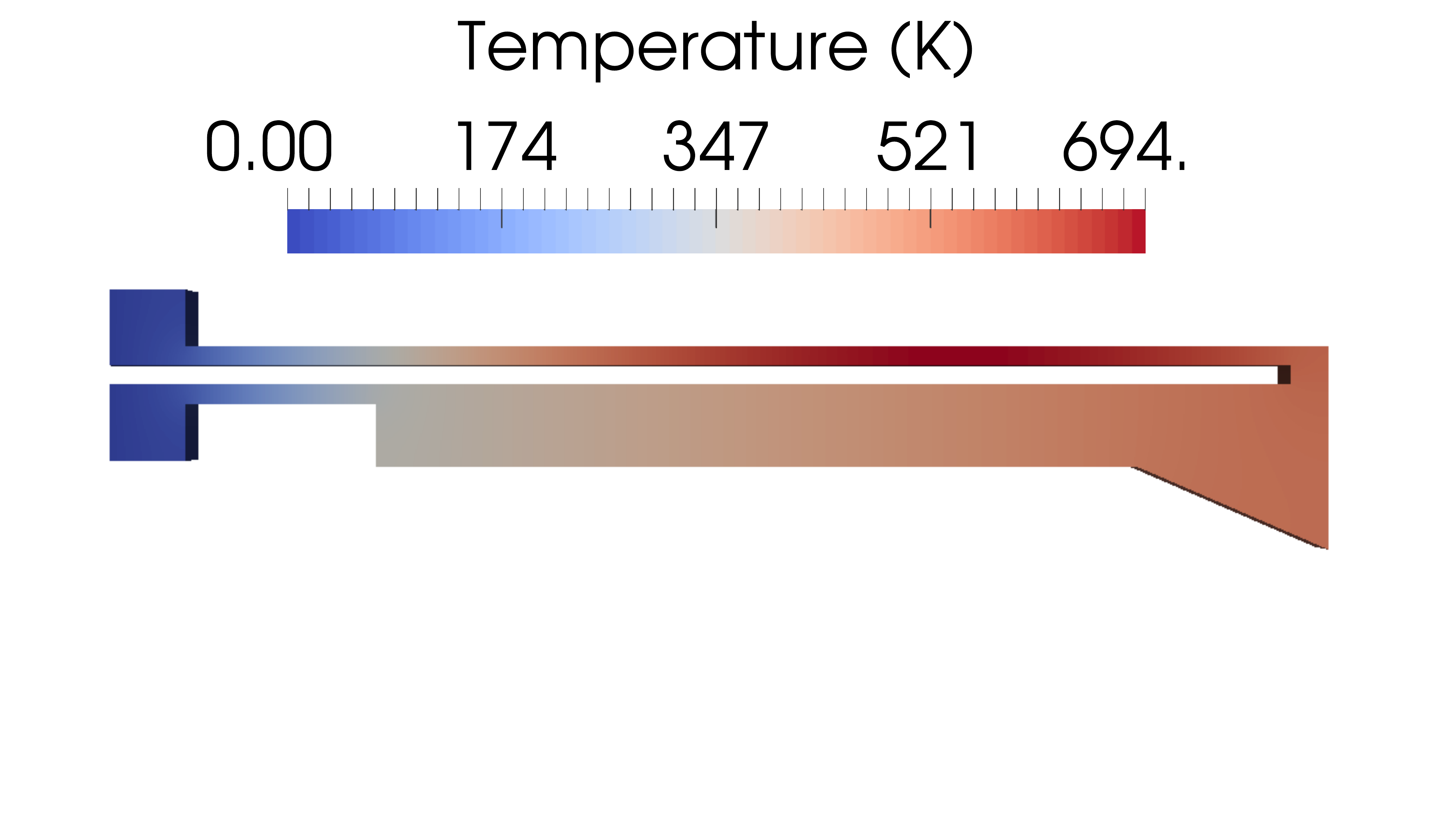}
\includegraphics[width=0.32\columnwidth]{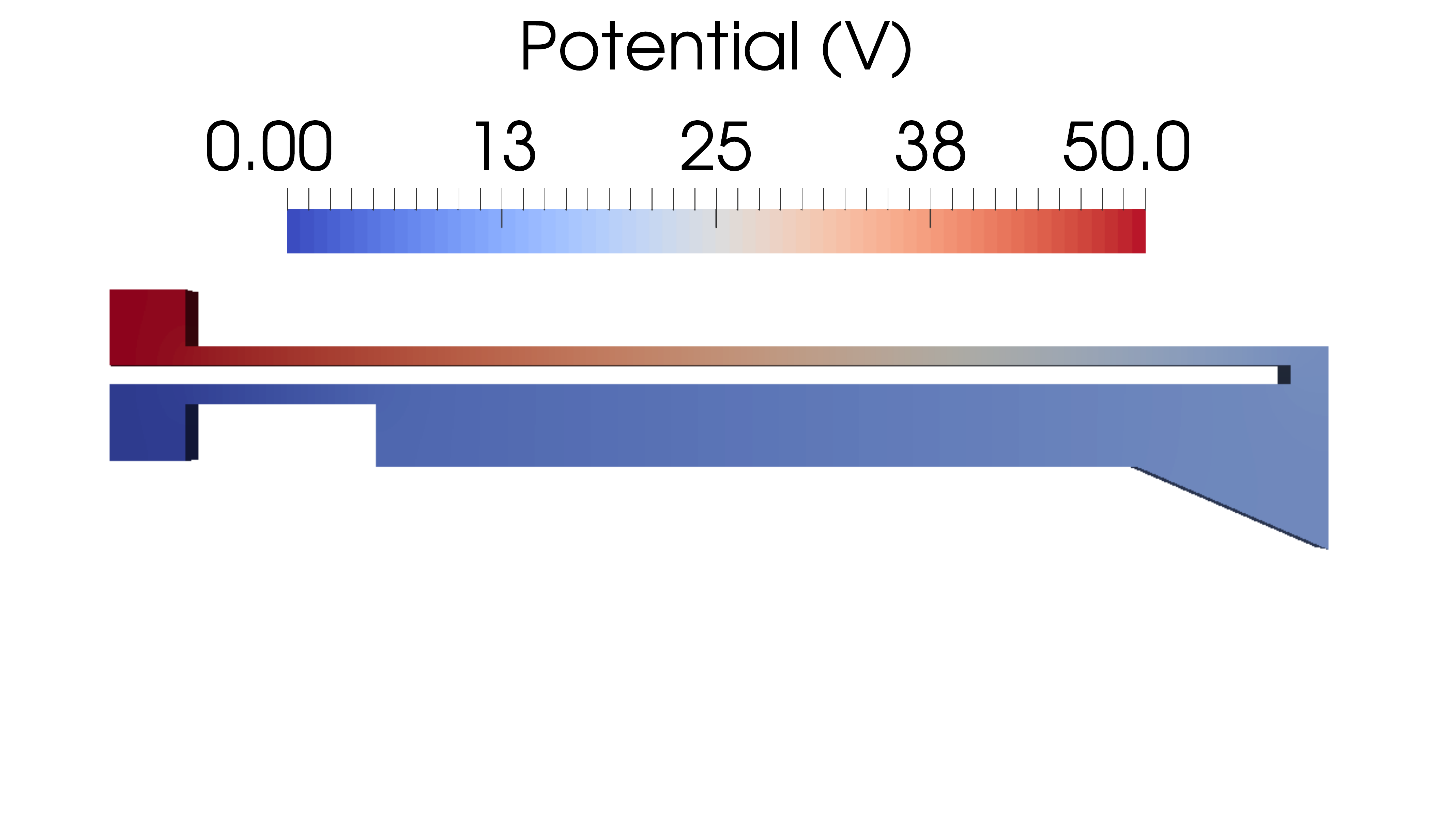}
\includegraphics[width=0.32\columnwidth]{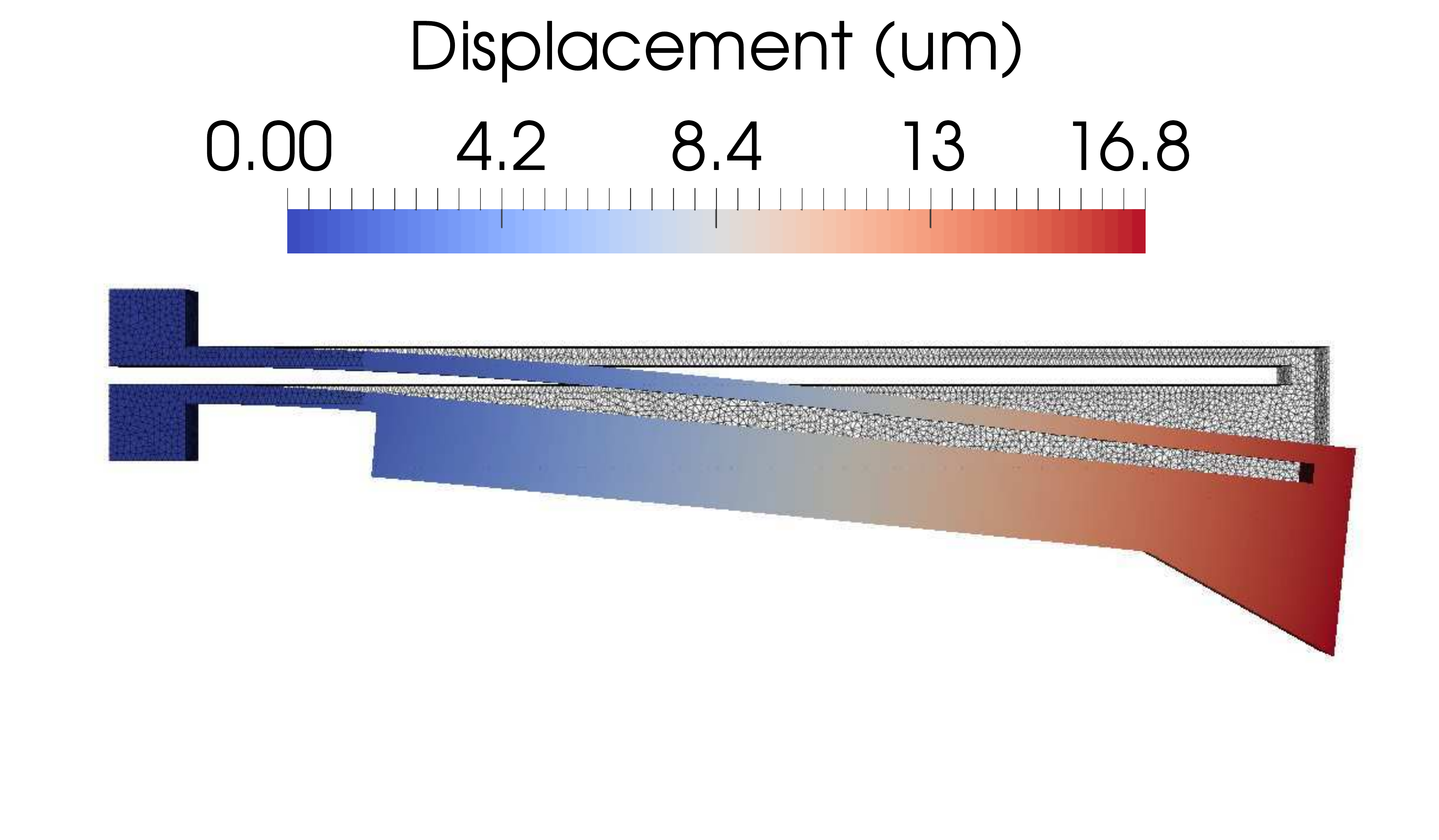}
\caption{The approximation to the solution of the problem defined in Section~\ref{subsec:problem3} at $t = \tend$ and with the finest spatial and temporal discretization. In the right-most plot, the grid has been deformed according to the computed displacement and then super-imposed over the original mesh to illustrate the deformation. We note that the grid is never deformed in the actual computations. (This figure is in color in the electronic version of the article.)}
\label{fig:MEMS_solutions}
\end{figure}

\FloatBarrier

\section{Conclusions and outlook} \label{sec:conclusions}
We have presented a fully discrete numerical method for the fully coupled thermoviscoelastic thermistor problem~\eqref{eq:main_T}--\eqref{eq:main_stress} and proved optimal convergence orders in both space and time. These theoretical results are validated by experimental results. 

We reiterate that mixed boundary conditions and re-entrant corners might lead to order reductions. In that case an adaptive mesh refinement strategy may be used, which requires a good a posteriori error estimate. It is possible that the ideas in~\cite{AllegrettoYan2006} regarding this can be extended to the present, deformable case.

As illustrated by Section~\ref{subsec:problem3}, a typical thermistor is not convex, so a further item that could be improved in the analysis is therefore the shape of the computational domain itself. In this direction we note that the stationary version of the non-deformable problem has been studied in~\cite{Holst_etal2010,JensenMalqvist2013} for very general domains. It is our ambition to extend these ideas to the time-dependent deformable case in the future.

Finally, a similar analysis would apply also for higher-order methods both in time and space. See e.g.~\cite{LiGaoSun2014} for a Crank-Nicolson-approach to the non-deformable Joule heating problem. However, such an analysis would require extra regularity assumptions that are unfeasible in real-world engineering applications.

\bibliographystyle{spmpsci}
\bibliography{references}

\end{document}